\documentclass[11pt]{scrartcl}
\usepackage[utf8]{inputenc}
\usepackage[T1]{fontenc}
\usepackage{lmodern}
\usepackage{alphabeta}

\usepackage{titling}
\usepackage{xspace}
\usepackage{soul} % for \st to strike out text

\usepackage[letterpaper, top=25.4mm, bottom=25.4mm, left=25.4mm, right=25.4mm, includefoot]{geometry}

\DeclareSectionCommand[%
level=4,
indent=0pt,
beforeskip=1ex plus 1ex minus .2ex,
afterskip=-1em,
font={},
tocindent=7em,
tocnumwidth=4.1em,
counterwithin=subsubsection
]{paragraph}

\usepackage[inline]{enumitem}

\usepackage{dsfont}
\usepackage{setspace}

\usepackage{bm}
\usepackage{microtype}
\usepackage{amsmath}
\usepackage{amssymb}
\usepackage{amsfonts}
\usepackage{mathtools}
\usepackage[mathscr]{euscript}

\usepackage[hyphens]{url}
\usepackage{amsthm}

\usepackage{tikz}
\usepackage{xcolor}
\usepackage{subcaption}
\usepackage{graphicx}
\usepackage{wrapfig}

\usepackage{hyperref}
\usepackage{zref-clever}

\usepackage{nicefrac}
\usepackage[textwidth=1.5in]{todonotes}

\colorlet{myGreen}{green!50!black}
\colorlet{myLightgreen}{green}
\colorlet{myRed}{red!90!black}
\definecolor{myBlue}{rgb}{0.25, 0.0, 1.0}
\definecolor{myLightBlue}{rgb}{0.39, 0.58, 0.93}
\colorlet{myViolet}{myBlue!55!myRed}
\definecolor{myOrange}{rgb}{1.0, 0.66, 0.07}

\definecolor{CornflowerBlue}{rgb}{0.39, 0.58, 0.93}
\definecolor{DarkGoldenrod}{rgb}{0.72, 0.53, 0.04}
\definecolor{BritishRacingGreen}{rgb}{0.0, 0.26, 0.15}
\definecolor{DarkMagenta}{rgb}{0.55, 0.0, 0.55}
\definecolor{AO}{rgb}{0.0, 0.5, 0.0}
\definecolor{BostonUniversityRed}{rgb}{0.8, 0.0, 0.0}
\definecolor{myRed}{rgb}{0.8, 0.0, 0.0}
\definecolor{DarkMidnightBlue}{rgb}{0.0, 0.2, 0.4}
\definecolor{DarkTangerine}{rgb}{1.0, 0.66, 0.07}
\definecolor{AppleGreen}{rgb}{0.55, 0.71, 0.0}
\definecolor{BrightUbe}{rgb}{0.82, 0.62, 0.91}
\definecolor{Amethyst}{rgb}{0.6, 0.4, 0.8}
\definecolor{DarkGray}{rgb}{0.52, 0.52, 0.51}
\definecolor{Gray}{rgb}{0.66, 0.66, 0.66}
\definecolor{BananaYellow}{rgb}{1.0, 0.88, 0.21}
\definecolor{Amber}{rgb}{1.0, 0.75, 0.0}
\definecolor{LightGray}{rgb}{0.83, 0.83, 0.83}
\definecolor{PrincetonOrange}{rgb}{1.0, 0.56, 0.0}
\definecolor{DeepCarrotOrange}{rgb}{0.91, 0.41, 0.17}
\definecolor{CarrotOrange}{rgb}{0.93, 0.57, 0.13}
\definecolor{MidnightBlue}{rgb}{0.1, 0.1, 0.44}
\definecolor{Magenta}{rgb}{0.50, 0.0, 0.50}
\definecolor{BrightPink}{rgb}{1.0, 0.0, 0.5}
\definecolor{BrilliantRose}{rgb}{1.0, 0.33, 0.64}
\definecolor{ChromeYellow}{rgb}{1.0, 0.65, 0.0}
\definecolor{HotMagenta}{rgb}{1.0, 0.11, 0.81}
\definecolor{Auburn}{rgb}{0.43, 0.21, 0.1}
\definecolor{BrightTurquoise}{rgb}{0.03, 0.91, 0.87}
\definecolor{DarkCyan}{rgb}{0.0, 0.55, 0.55}
\definecolor{Salmon}{rgb}{1.0, 0.55, 0.41}

\vfuzz \hfuzz
\setlist[itemize]{topsep=0pt,partopsep=0pt,itemsep=0pt,parsep=0pt}
\setlist[itemize,1]{label={\small\textbullet}}
\setlist[itemize,2]{label={\tiny\textbullet}}
\setlist[itemize,3]{label=$\cdot$}
\setlist[enumerate]{topsep=0pt,partopsep=0pt,itemsep=0pt,parsep=0pt}
\setlist[enumerate,1]{label=\roman*)}
\setlist[enumerate,2]{label=\alph*)}
\setlist[enumerate,3]{label=\arabic*)}

\hypersetup{
colorlinks=true,
linkcolor=AO!65!black,
citecolor=AO!65!black,
urlcolor=AppleGreen!65!black,
bookmarksopen=true,
bookmarksnumbered,
bookmarksopenlevel=2,
bookmarksdepth=3
}

\newcommand*\samethanks[1][\value{footnote}]{\footnotemark[#1]}

\makeatletter
\newtheorem*{rep@theorem}{\rep@title}
\newcommand{\newreptheorem}[2]{%
\newenvironment{rep#1}[1]{%
 \def\rep@title{#2~\ref{##1}}%
 \begin{rep@theorem}}%
 {\end{rep@theorem}}}
\makeatother

\newreptheorem{theorem}{Theorem}
\newreptheorem{lemma}{Lemma}

\theoremstyle{definition}

\newtheorem{environment}{Environment}[section]

\newtheorem{lemma}[environment]{Lemma}
\AddToHook{env/lemma/begin}{\zcsetup{countertype={environment=lemma}}}
\zcRefTypeSetup{lemma}{
Name-sg = Lemma ,
name-sg = Lemma ,
Name-pl = Lemmas ,
name-pl = Lemmas ,
}

\newtheorem*{lemma*}{Lemma}
\AddToHook{env/lemma*/begin}{\zcsetup{countertype={environment=lemma*}}}
\zcRefTypeSetup{lemma*}{
Name-sg = Lemma ,
name-sg = Lemma ,
Name-pl = Lemmas ,
name-pl = Lemmas ,
}

\AddToHook{env/proposition/begin}{\zcsetup{countertype={environment=proposition}}}
\zcRefTypeSetup{proposition}{
Name-sg = Proposition ,
name-sg = Proposition ,
Name-pl = Propositions ,
name-pl = Propositions ,
}

\newtheorem{corollary}[environment]{Corollary}
\AddToHook{env/corollary/begin}{\zcsetup{countertype={environment=corollary}}}
\zcRefTypeSetup{corollary}{
Name-sg = Corollary ,
name-sg = Corollary ,
Name-pl = Corollaries ,
name-pl = Corollaries ,
}

\newtheorem{theorem}[environment]{Theorem}
\AddToHook{env/theorem/begin}{\zcsetup{countertype={environment=theorem}}}
\zcRefTypeSetup{theorem}{
Name-sg = Theorem ,
name-sg = Theorem ,
Name-pl = Theorems ,
name-pl = Theorems ,
}

\newtheorem*{theorem*}{Theorem}
\AddToHook{env/theorem*/begin}{\zcsetup{countertype={environment=theorem*}}}
\zcRefTypeSetup{theorem*}{
Name-sg = Theorem ,
name-sg = Theorem ,
Name-pl = Theorems ,
name-pl = Theorems ,
}

\newtheorem{conjecture}[environment]{Conjecture}
\AddToHook{env/conjecture/begin}{\zcsetup{countertype={environment=conjecture}}}
\zcRefTypeSetup{conjecture}{
Name-sg = Conjecture ,
name-sg = Conjecture ,
Name-pl = Conjectures ,
name-pl = Conjectures ,
}

\newtheorem*{hypothesis*}{Hypothesis}
\AddToHook{env/hypothesis*/begin}{\zcsetup{countertype={environment=hypothesis*}}}
\zcRefTypeSetup{hypothesis*}{
Name-sg = Hypothesis ,
name-sg = Hypothesis ,
Name-pl = Hypotheses ,
name-pl = Hypotheses ,
}

\newtheorem{observation}[environment]{Observation}
\AddToHook{env/observation/begin}{\zcsetup{countertype={environment=observation}}}
\zcRefTypeSetup{observation}{
Name-sg = Observation ,
name-sg = Observation ,
Name-pl = Observations ,
name-pl = Observations ,
}

\AddToHook{env/example/begin}{\zcsetup{countertype={environment=example}}}
\zcRefTypeSetup{example}{
Name-sg = Example ,
name-sg = Example ,
Name-pl = Examples ,
name-pl = Examples ,
}

\AddToHook{env/remark/begin}{\zcsetup{countertype={environment=remark}}}
\zcRefTypeSetup{remark}{
Name-sg = Remark ,
name-sg = Remark ,
Name-pl = Remarks ,
name-pl = Remarks ,
}

\zcRefTypeSetup{equation}{
Name-sg = Equation ,
name-sg = Equation ,
Name-pl = Equations ,
name-pl = Equations ,
}

\zcRefTypeSetup{chapter}{
Name-sg = Chapter ,
name-sg = Chapter ,
Name-pl = Chapters ,
name-pl = Chapters ,
}

\zcRefTypeSetup{section}{
Name-sg = Section ,
name-sg = Section ,
Name-pl = Sections ,
name-pl = Sections ,
}

\zcRefTypeSetup{appendix}{
Name-sg = Appendix ,
name-sg = Appendix ,
Name-pl = Appendices ,
name-pl = Appendices ,
}

\zcRefTypeSetup{algorithm}{
Name-sg = Algorithm ,
name-sg = Algorithm ,
Name-pl = Algorithms ,
name-pl = Algorithms ,
}

\AddToHook{env/notation/begin}{\zcsetup{countertype={environment=notation}}}
\zcRefTypeSetup{notation}{
Name-sg = Notation ,
name-sg = Notation ,
Name-pl = Notations ,
name-pl = Notations ,
}

\newtheorem{question}[environment]{Question}
\AddToHook{env/question/begin}{\zcsetup{countertype={environment=question}}}
\zcRefTypeSetup{question}{
Name-sg = Question ,
name-sg = Question ,
Name-pl = Questions ,
name-pl = Questions ,
}

\AddToHook{env/problem/begin}{\zcsetup{countertype={environment=problem}}}
\zcRefTypeSetup{problem}{
Name-sg = Problem ,
name-sg = Problem ,
Name-pl = Problems ,
name-pl = Problems ,
}

\AddToHook{env/claim/begin}{\zcsetup{countertype={environment=claim}}}
\zcRefTypeSetup{claim}{
Name-sg = Claim ,
name-sg = Claim ,
Name-pl = Claims ,
name-pl = Claims ,
}

\AddToHook{env/definition/begin}{\zcsetup{countertype={environment=definition}}}
\zcRefTypeSetup{definition}{
Name-sg = Definition ,
name-sg = Definition ,
Name-pl = Definitions ,
name-pl = Definitions ,
}

\zcRefTypeSetup{figure}{
Name-sg = Figure ,
name-sg = Figure ,
Name-pl = Figures ,
name-pl = Figures ,
}

\usetikzlibrary{calc}
\usetikzlibrary{fit}
\usetikzlibrary{decorations}
\usetikzlibrary{decorations.pathmorphing}
\usetikzlibrary{decorations.text}
\usetikzlibrary{shapes,hobby}

\tikzset{
	position/.style args={#1:#2 from #3}{
		at=($(#3)+(#1:#2)$)
	}
}

\tikzset{
  v:main/.style = {draw, circle, scale=0.8, thick,fill=black,inner sep=0.7mm},
  v:ghost/.style = {inner sep=0pt,scale=1},
  >={latex},
  e:marker/.style = {line width=8.5pt,line cap=round,opacity=0.35,color=DarkGoldenrod},
  e:main/.style = {line width=1pt},
}

\definecolor{orchid}{RGB}{180,100,163}

\newcommand{\EP}{Erd\H{o}s-P\'osa}

\usepackage{makecell}

\title{An Erd\H{o}s-P\'osa theorem for cycles and faces of distinct lengths
\thanks{\href{mailto:pascal.gollin@famnit.upr.si}{pascal.gollin@famnit.upr.si},
\href{mailto:m.gorsky@pm.me}{m.gorsky@pm.me},
\href{mailto:research@meikehatzel.com}{research@meikehatzel.com},
\href{mailto:kevin.hendrey1@monash.edu}{kevin.hendrey1@monash.edu},
\href{mailto:tony@ibs.re.kr}{tony@ibs.re.kr},
\href{mailto:cmcfarland30@gatech.edu}{cmcfarland30@gatech.edu},
\href{mailto:msokolow@mpi-inf.mpg.de}{msokolow@mpi-inf.mpg.de},
\href{mailto:wiederrecht@kaist.ac.kr}{wiederrecht@kaist.ac.kr},
\href{mailto:wollan@di.uniroma1.it}{wollan@di.uniroma1.it}.}}
\predate{}
\date{}
\postdate{}

\preauthor{}
\DeclareRobustCommand{\authorthing}{
	\begin{center}
        \begin{tabular}{ccc}
        J.~Pascal Gollin\thanks{FAMNIT, University of Primorska, Koper, Slovenia. Supported by the Slovenian Research and Innovation Agency (research project N1-0370).} &
        Maximilian Gorsky\thanks{Discrete Mathematics Group, Institute for Basic Science (IBS), Daejeon, South Korea. Supported by the Institute for Basic Science (IBS-R029-C1).} &
        Meike Hatzel\thanks{Technical University Darmstadt, Darmstadt, Germany} \\
        Kevin Hendrey\thanks{School of Mathematics, Monash University, Melbourne, Australia. Supported by the Australian Research Council} &
        Tony Huynh\samethanks[3] &
        Caleb McFarland\thanks{School of Mathematics, Georgia Institute of Technology, Atlanta, USA. Supported in part by the Georgia Tech ARC-ACO Fellowship and in part by the National Science Foundation under Grant No.~DMS-2452111.} \\
        Marek Sokołowski\thanks{Max Planck Institute for Informatics, Saarland Informatics Campus, Saarbr\"{u}cken, Germany} &
		Sebastian Wiederrecht\thanks{School of Computing, KAIST, Daejeon, South Korea} &
        Paul Wollan\thanks{Sapienza University of Rome, Rome, Italy} 
    \end{tabular}
\end{center}}
\author{\authorthing}
\postauthor{}

\begin{document}
\maketitle

\begin{abstract}
We show that for every $k \in \mathbb{N}$, every graph $G$ contains $k$ vertex-disjoint cycles of different lengths, or there exists a set $X \subseteq V(G)$ with $|X| \in \mathcal{O}(k^6\mathsf{polylog}(k))$ such that $G-X$ has at most $k-1$ cycle lengths.   

We also prove analogous results for facial lengths of embedded graphs.
Let $G$ be a graph with a closed 2-cell embedding $\psi$ on a surface $\Sigma$ of Euler genus $g$, let $c$ be a colouring of the faces $\mathcal{F}(\psi)$ of $\psi$, and let $R(G,\psi)$ be the radial graph of $(G, \psi)$.
Then there exist $k$ faces $F_1, \ldots , F_k \in \mathcal{F}(\psi)$ that are given pairwise distinct colours by $c$ and are pairwise at distance at least $d$ in $\psi$, or there exists a set $X \subseteq V(G)$ of order at most $\mathcal{O}(k^2dg)$ such that $|\{ c(F) \mid F \in \mathcal{F}(\psi) \text{ and } V(F) \cap \bigcup_{x \in X} N^d_{R(G,\psi)}(x) = \emptyset \}| \leq k(k+2)$.

Finally, using a result from additive combinatorics, we show that there are subdivided ladders with only a small number of cycle lengths.
This suggests that it may be difficult to improve our bounds.  
\end{abstract}
\let\sc\itshape
\thispagestyle{empty}

\newpage

\setcounter{page}{1}

\section{Introduction} \label{sec:intro}
The main problem we address in this paper is whether a graph $G$ contains many vertex-disjoint cycles of different lengths.  Bensmail, Harutyunyan, Le, Li, and Lichiardopol~\cite{BHLLL17} proved that the answer is yes for graphs with large minimum degree.   

\begin{theorem}[Bensmail et al.\ \cite{BHLLL17}] \label{thm:mindegree}
    For every $k \geq 1$, every graph of minimum degree at least $\frac{5k^2+5k-2}{2}$ contains $k$ vertex-disjoint cycles of different lengths.  
\end{theorem}

Rather than a sufficient condition, we instead prove a rough structure theorem for the set of graphs which do not contain many vertex-disjoint cycles of different lengths.  The related problem of when a graph contains many vertex-disjoint \emph{induced} cycles of different lengths was very recently considered by Chudnovsky and Maier~\cite{chudnovsky2026induced}. Our theorem can be viewed as an \emph{Erd\H{o}s-P\'osa} type theorem, which is a reference to the following classic result by Erd\H{o}s and P\'osa.

\begin{theorem}[Erd\H{o}s-P\'osa Theorem~\cite{ErdosP1965Independent}]\label{thm:ep}
    There exists a function $f(k) \in \mathcal{O}(k \log k)$ such that the following holds. For every $k \in \mathbb{N}$ and every graph $G$, $G$ contains $k$ vertex-disjoint cycles, or there exists a set $X \subseteq V(G)$ with $|X| \leq f(k)$ such that $G-X$ is a forest.   
\end{theorem}

The Erd\H{o}s-P\'osa Theorem has been hugely influential in graph structure theory and has been studied in numerous other settings, including minors~\cite{RobertsonS1986Graph, CamesvanBatenburgHJR2019Tight,PaulPTW2024Obstructions}, topological minors~\cite{Thomassen1988Presence, Liu2022Packing,PaulPTW2024Obstructions}, immersions~\cite{Liu2021Packing, KK18}, directed cycles~\cite{ReedRST1996Packing}, matroid circuits~\cite{GK09}, and vertex minors~\cite{EPvertexminors}.

Moreover, there is an extensive line of work on the Erd\H{o}s-P\'osa property for cycles satisfying various constraints including triangles~\cite{tuza90}, odd cycles~\cite{Reed1999Mangoes,KawarabayashiKKX2025Halfintegral}, even cycles~\cite{BruhnHJ2018Frames,GorskyKKW2024Packing,Gorsky2024Structure}, cycles of length $0 \pmod{m}$~\cite{Thomassen1988Presence},  cycles of length $\ell \pmod{p}$ with $p$ prime~\cite{ThomasY2023Packinga}, long cycles~\cite{BBR07, CJU20, MNSW17}, holes~\cite{KimK2020ErdHosPosa, HK24}, $S$-cycles \cite{KakimuraKM2011Packing}, and $(S_1,S_2)$-cycles \cite{HuynhJW2019Unified}.  Most (but not all) of the aforementioned results on cycles now follow from very general theorems for cycles in graphs with edge-labels from multiple abelian groups~\cite{HuynhJW2019Unified, GollinHKOY2022Unified, GollinHKKO2024Unified}.  For more details, we refer the interested reader to~\cite{RaymondT2017Recent} for a general survey on Erd\H{o}s-P\'osa theorems.  

All of the above results can be stated in terms of packings and covers in hypergraphs, as we now describe.   
Given a hypergraph $H$, a \emph{packing} is a set of disjoint hyperedges, and a \emph{cover} is a set of vertices $X$ such that $e \cap X \neq \emptyset$ for all $e \in E(H)$.  We define $\nu(H)\coloneqq\max\{|M| : \text{$M$ is a packing}\}$ and $\tau(H)\coloneqq\min\{ |X| : \text{$X$ is a cover}\}$.  Clearly, $\tau(H) \geq \nu(H)$ for every hypergraph $H$.  We say that a family $\mathcal{H}$ of hypergraphs satisfies the \emph{Erd\H{o}s-P\'osa property} if there exists a function $f$ such that $\tau(H) \leq f(\nu(H))$ for every $H \in \mathcal{H}$.
For example, for every graph $G$, we can make a hypergraph $H_G$, where $V(H_G)=V(G)$ and $E(H_G)=\{V(C) \mid \text{$C$ is a cycle of $G$}\}$. The original theorem of Erd\H{o}s and P\'osa is simply the assertion that the family of hypergraphs $\mathcal{H}\coloneqq\{H_G \mid \text{$G$ is a graph}\}$ has the Erd\H{o}s-P\'osa property.  Moreover, maximum packings and minimum covers are optimal integral solutions to two dual linear programs, and the Erd\H{o}s-P\'osa function $f$ measures the integrality gap between the two linear programs.  

% \CMin{The below is not necessarily true: you can construct a hypergraph with vertex set $V(G) \cup \mathbb{N}$ and edge set $\{V(C) \cup \{\ell\} : C \text{ is a cycle in } G \text{ of length }\ell\}$. Then a cover gives you a hitting set of vertices and allowed lengths. Perhaps we should say something like ``in a hypergraph over $V(G)$''.}

% \THin{Good point.  I changed it as suggested.}
Interestingly, maximising the number of vertex-disjoint cycles of different lengths in a graph $G$ is not a maximum packing problem in a hypergraph with vertex set $V(G)$.  Thus, our main problem does not fit into the above framework.  Nonetheless, we provide another type of ``dual certificate'' whenever a graph does not contain $k$ vertex-disjoint cycles of different lengths.  Our dual certificate is stated in terms of the \emph{cycle spectrum} $\Lambda(G)$ of a graph $G$, defined as $\Lambda(G)\coloneqq\{ \ell \in \mathbb{N} \mid \text{$G$ contains a cycle of length $\ell$}\}$.  The cycle spectrum of a graph has been extensively studied (see the surveys~\cite{cyclespectrum,BGS22}). For example, Bondy's famed meta-conjecture~\cite{bondy71} asserts that every non-trivial condition implying that a graph $G$ is Hamiltonian actually implies that $\Lambda(G)=\{3, \dots, |V(G)|\}$ (with some simple family of exceptions).  

The following is our first main result.  

\begin{theorem} \label{thm:main}
   For every $k \in \mathbb{N}$, every graph $G$ contains $k$ vertex-disjoint cycles of different lengths, or there exists a set $X \subseteq V(G)$ with $|X| \in \mathcal{O}(k^6\mathsf{polylog}(k))$ such that $|\Lambda(G-X)| \leq k-1$. 

    Furthermore, there is an algorithm that given $k \in \mathbb{N}$ and a graph $G$ as input, either returns $k$ vertex-disjoint cycles of different lengths in $G$, or $X$ as above, in time $\mathcal{O}(18^{k^5}\mathsf{poly}(k)|V(G)|^7)$.
    % $2^{\mathcal{O}(k^2)}|V(G)|^6$.
\end{theorem}

Note that for every $X \subseteq V(G)$, the number of vertex-disjoint cycles of distinct lengths is at most $|\Lambda(G-X)|+|X|$.   In particular, the set $X$ from \zcref{thm:main} certifies that $G$ contains $\mathcal{O}(k^6\mathsf{polylog}(k))$ vertex-disjoint cycles of different lengths.  In \zcref{sec:lowerbounds} we show that the bound $|\Lambda(G-X)| \leq k-1$ in \zcref{thm:main} is best possible even for arbitrarily large $X$.  

We also prove analogous results for facial cycles of embedded graphs.  In fact, we prove a more general theorem where faces are ``coloured'' (not necessarily by their length), and we seek facial cycles with distinct colours that are also pairwise far apart.  
% Our theorem is most conveniently stated in terms of \emph{$\Sigma$-map graphs} (see~\cite{DEW17}), which generalise graphs embedded on $\Sigma$.  
Precise definitions will be given later in~\zcref{sec:faces}.

\begin{theorem}\label{thm:facesatdistance}
      Let $k,d \in \mathbb{N}^+$ and $g \in \mathbb{N}$.
    Let $G$ be a graph with a 2-cell embedding $\psi$ on a surface $\Sigma$ of Euler genus $g$, let $c$ be a colouring of the faces $\mathcal{F}(\psi)$ of $\psi$, and let $R(G,\psi)$ be the radial graph of $(G, \psi)$.

    Then there exist $k$ faces $F_1, \ldots, F_k \in \mathcal{F}(\psi)$ that are given pairwise distinct colours by $c$ and are pairwise at distance at least $d$ in $\psi$, or there exists a set $X \subseteq V(G)$ of order at most $\mathcal{O}(k^2dg)$ such that $|\{ c(F) \mid F \in \mathcal{F}(\psi) \text{ and } V(F) \cap \bigcup_{x \in X} N^d_{R(G,\psi)}(x) = \emptyset \}| \leq k(k+2)$.

    Moreover, there exists an algorithm running in time $2^{\mathcal{O}(d^2g^2)}\mathsf{poly}(k)|V(G)|^4$ that finds either the faces $F_1, \ldots, F_k$ or the set $X$ with the properties promised above.
\end{theorem}

Theorems~\ref{thm:main} and~\ref{thm:facesatdistance} are rare instances of Erd\H{o}s-P\'osa theorems which cannot be naturally phrased as a maximum packing versus minimum cover problem in a hypergraph.  Other notable examples include the recent Erd\H{o}s-P\'osa theorems in the emerging field of ``coarse graph theory''~\cite{AHJKW24, GP25, AGHK25, DJMM24}.  

We achieve our bounds for the size of the hitting set in~\zcref{thm:main} by analysing the cycle spectrum of ladder-like graphs, whose presence we can guarantee given high enough treewidth.
Using a construction from additive combinatorics, we also prove that there are subdivided ladders whose cycle spectrum is small.  
In the following statement, the \emph{$k$-ladder} is defined as the Cartesian product of $K_2$ and a~$k$-vertex path.

\begin{theorem}
    \label{lem:ladder-subdivision-with-few-cycles}
    For every $n \in \mathbb{N}$ and $\alpha > \frac{\log 2}{\log (1 + \sqrt{2})} \approx 0.7865$, there exists $k \geq n$ and a~subdivision $G$ of the $k$-ladder such that $|\Lambda(G)| \leq k^\alpha$.
\end{theorem}

\zcref{lem:ladder-subdivision-with-few-cycles} suggests that different techniques are required to improve the bound in~\zcref{thm:main}.

\textbf{Paper Outline.} We begin by proving the non-algorithmic version of~\zcref{thm:main} in~\zcref{sec:distinct}, with the necessary arguments for the algorithmic part being deferred to \zcref{sec:algoproof}. In~\zcref{sec:lowerbounds}, we provide constructions giving lower bounds for \zcref{thm:main}. We then prove~\zcref{thm:facesatdistance} in~\zcref{sec:faces} and \zcref{lem:ladder-subdivision-with-few-cycles} in \zcref{sec:bounds}.
Finally, in \zcref{sec:extensions}, we discuss possible extensions of our main results, focusing primarily on counterexamples across various settings.

\section{Disjoint cycles with distinct lengths}\label{sec:distinct}
The proofs in this section rely heavily on the graph minor structure theory of Robertson and Seymour, in particular on treewidth and arguments on tree-decompositions.

Let us first define what we mean by a tree-decomposition.
A \emph{tree-decomposition} of a graph $G$ is a pair $\mathcal{T} = (T, \beta)$, where $T$ is a tree and $\beta \colon V(T) \rightarrow 2^{V(G)}$ such that
\begin{itemize}
    \item for every edge $e \in E(G)$, there exists a $t \in V(T)$ such that $e \subseteq \beta(t)$, and
    \item for every vertex $v \in V(G)$, the set $\beta^{-1}(v)$ induces a connected, non-empty subtree of $T$.
\end{itemize}
The \emph{width} of $\mathcal{T}$ is the maximum value of $|\beta(t)| - 1$ over all $t \in V(T)$ and the \emph{treewidth} $\mathsf{tw}(G)$ of $G$ is the minimum width over all tree-decompositions of $G$.

If $T$ has a designated root $r$, we call $(T,r,\beta)$ a \emph{rooted} tree-decomposition.
Given a tree $T$ with a root $r$ and a vertex $t \in V(T)$, we denote by $T_t$ the subtree of $T$ that contains all vertices whose path to $r$ in $T$ contains $t$.
Further, if $(T,r,\beta)$ is a rooted tree-decomposition of a graph $G$, we denote by $G_t$ the subgraph $G[\bigcup_{s \in V(T_t)} \beta(s)]$.

We first show that graphs with bounded treewidth have our desired packing-covering property for cycles of distinct lengths.
This is derived from the following lemma.

\begin{lemma}\label{lem:rainbowtreepacking}
    Let $k \in \mathbb{N}^+$. Let $T$ be a tree and $T_1, \dots, T_m$ be non-empty subtrees of $T$. Let $c$ be a colouring of the subtrees $T_1, \dots, T_m$. Then there exists $k$ vertex-disjoint subtrees each given distinct colours by $c$, or there exists a set $X \subseteq V(T)$ with $|X| \leq k-1$ such that there are at most $k-1$ colours among the subtrees disjoint from $X$.
\end{lemma}
\begin{proof}
    We proceed by induction on $|V(T)|$. Clearly, the result holds for $|V(T)| = 1$. Suppose $|V(T)| \geq 2$ and let $v \in V(T)$ be a leaf. 
    If there exists an ${i \in [m]}$ such that $V(T_i) = \{v\}$, then we add $T_i$ to our collection. 
    We then update our collection by removing all subtrees which have colour $c(T_i)$ or contain the vertex $v$, and we apply induction to $T - v$. If we find $k-1$ disjoint subtrees of distinct colours in this updated collection, then together with $T_i$, this is the desired collection. If we find a set $X' \subseteq V(T - v)$ of at most $k-2$ vertices hitting all subtrees with colours outside a set of at most $k-2$ colours, then $X = X' \cup \{v\}$ is the desired hitting set.

    If none of the subtrees~$T_1, \dots, T_m$ contains only the vertex $v$, then we delete $v$ from every subtree and apply induction to $T - v$.  
    A hitting set $X$ in $T-v$ is also a hitting set in $T$. 
    On the other hand, disjoint subtrees in $T-v$ lift to disjoint subtrees in $T$. 
    This completes the proof.
\end{proof}

Applying the above lemma to an optimal tree-decomposition of $G$ yields the following corollary.

\begin{lemma}\label{lem:lowtwdistinctlengthep}
    Let $k \in \mathbb{N}^+$ and let $G$ be a graph of treewidth at most $\ell$.
    Then $G$ contains $k$ vertex-disjoint cycles of distinct lengths, or there exists a set $X \subseteq V(G)$ with $|X| \leq (k-1)(\ell+1)$ such that $|\Lambda(G - X)| \leq k-1$.
\end{lemma}
% \CMin{is it worth writing the much more general statement that if $G$ has treewidth at most $\ell$ and $H_1, \dots, H_m$ is any set of coloured connected subgraphs, then either we can find $k$ disjoint subgraphs of pairwise different colours or a set $X$ with $|X| \leq (k-1)(\ell +1)$ such that there are at most $k-1$ colours among the unhit subgraphs? \zcref{lem:lowtwdistinctlengthep} is applying this for $H_i$ the set of cycles coloured by their lengths, and \zcref{lem:bndtwcolouredfacedisjointep} is (the algorithmic version of) applying this for $H_i$ the balls around faces.}
% \MGin{I would like to avoid adding more complexity to this paper. The argument we present here is also of a fairly standard flavour.}
\begin{proof}
    Let $(T, \beta)$ be a tree-decomposition of $G$ of width at most $\ell$. Let $C_1, \dots, C_m$ be a list of all the cycles in $G$. For each cycle $C_i$, let $T_i$ be the subtree of $T$ with vertex set $\bigcup_{v \in V(C_i)} \beta^{-1}(v)$. Let $c(T_i)$ be the length of the cycle $C_i$. By~\zcref{lem:rainbowtreepacking}, there exist $k$ vertex-disjoint subtrees each given a distinct colour by $c$, or there exists $X' \subseteq V(T)$ of size at most $k-1$ such that there are at most $k-1$ colours among the subtrees disjoint from $X'$. If the former holds, then $G$ contains $k$ vertex-disjoint cycles of distinct lengths. If the latter holds, then $X\coloneqq\bigcup_{t \in X'} \beta(t)$ is the desired hitting set in $G$.
\end{proof}

Our next goal is to prove that in graphs of high treewidth, we can find a subgraph certifying the existence of many vertex-disjoint cycles of distinct lengths. One approach is to use the grid theorem of \cite{RobertsonS1986Grapha} (see \cite{ChuzhoyT2021Tighter} for the currently best-known bounds), which shows that a graph with high treewidth must contain a subdivision of a large \emph{wall}.  It is not too difficult to show that every subdivision of a sufficiently large wall contains many vertex-disjoint cycles of distinct lengths.  

However, to obtain better bounds, we bypass using the grid theorem and aim for a different subgraph certifying the existence of many vertex-disjoint cycles of distinct lengths.
Let $k \in \mathbb{N}^+$. The \emph{$k$-prism} is the Cartesian product of $K_2$ and the cycle of length $k$.
Birmel\'{e}, Bondy, and Reed~\cite{prisms} showed that every graph of treewidth $\Omega(k^2)$ contains a subdivision of the $k$-prism.

% It is not hard to show that subdivisions\footnote{Given two graphs $G$ and $H$, we say that $H$ can be found as a \emph{subdivision} in $G$ if there exists an $H' \subseteq G$ such that $H'$ can be constructed from $H$ by replacing the edges of $H$ by non-trivial paths which pairwise do not share any vertex except for possibly their endpoints, which are retained from the edges they replace.} of \emph{walls} have many disjoint cycles of distinct lengths, which are graphs known to have high treewidth and for being able to be found in any graph of high enough treewidth \cite{RobertsonS1986Grapha} (see \cite{ChuzhoyT2021Tighter} for the currently best known bounds). In fact, even subdivisions of \emph{ladders} have many disjoint cycles of distinct lengths, for which there exist better bounds on the required treewidth \cite{HuynhJMSW2022Excluding}. However, to obtain the best bounds, we use the following results.

% We first define the type of subgraph we are interested in finding.

\begin{theorem}[Birmel\'{e}, Bondy, and Reed \cite{prisms}]\label{thm:prisms}
    Every graph of treewidth at least $60k^2-120k+62$ contains a subdivision of the $k$-prism.
\end{theorem}

Next, we show that every subdivision of a large prism contains many (not necessarily disjoint) cycles of distinct lengths.

\begin{lemma}\label{lem:prismsubdivdifferentlengths}
    Let $k \in \mathbb{N}^+$ and $G$ be a subdivision of the $6k$-prism. Then $|\Lambda(G)| \geq k$.
\end{lemma}
\begin{proof}
    % Let $H$ be the twisted $36k^2$-prism. Let $C_1, C_2$ be the two cycles of length $36k^2$ which are joined by a matching, and enumerate their vertices as $V(C_1) = \{v_1, \dots, v_{36k^2}\}$ and $V(C_2) = \{u_1, \dots, u_{36k^2}\}$ in clockwise order. Let $a_1, \dots, a_{36k^2}$ be such that $a_i = j$ whenever $v_i$ is matched to $u_j$. By applying \zcref{thm:ErdosSzekeres}, we obtain a subsequence of $(a_i)_{i \in [36k^2]}$ of size at least $6k$ which is strictly increasing or decreasing. By deleting all matching edges not in this subsequence, we obtain a subdivision of a $6k$-prism. Thus, it suffices to prove the result for $G$ a subdivision of a $6k$-prism.

    Let $H$ be the $6k$-prism and $w: E(H) \to \mathbb{N}^+$.  It suffices to prove that $H$ has at least $k$ different cycle weights, where the weight of a cycle $C$ in $H$ is $w(C)\coloneqq\sum_{e \in E(C)} w(e)$. 
    Let $C_1, C_2$ be the two cycles of $H$ such that $V(C_1) = \{a_1, \dots, a_{6k}\}$, $V(C_2) = \{b_1, \dots, b_{6k}\}$, and $a_i$ is adjacent to $b_i$ for all $i \in [6k]$. For each $i \in [3k]$, let $C^i$ be the cycle in $H$ with vertices $a_{2i-1}, a_{2i}, b_{2i}, b_{2i-1}$. Partition $[3k]$ into $I^+ \uplus I^- \uplus I^0$, according to whether $w(C_1 \triangle C^i)-w(C_1)$ is positive, negative, or zero. Suppose $|I^+| \geq k$. Let $A \subset B \subseteq I^+$.  Note that $w(C_1 \triangle \bigcup_{i \in A} C^i) < w(C_1 \triangle \bigcup_{i \in B} C^i)$. Thus, $H$ contains $k$ different cycle weights.  Similarly, if $|I^-| \geq k$, then $H$ contains $k$ different cycle weights. Thus, we may assume $|I^0| \geq k$. But now $w(C_2 \triangle C^i) -w(C_2) > 0$ for all $i \in I^0$. By interchanging the roles of $C_1$ and $C_2$, we are done by the previous argument.  
\end{proof}

Finally, we need a result of Chekuri and Chuzhoy~\cite{chekuriC2013largetreewidth} on high treewidth partitions.  

\begin{theorem}[Chekuri and Chuzhoy {\cite[Theorem 1.1]{chekuriC2013largetreewidth}}]\label{thm:largetreewidthpartition}
    Let $G$ be any graph with $\mathsf{tw}(G) = k$. Let $h,r$ be any integers with $hr^2 \leq k/\mathsf{polylog}(k)$. Then $G$ contains vertex-disjoint subgraphs $G_1, \dots, G_h$ such that $\mathsf{tw}(G_i) \geq r$ for each $i$.
\end{theorem}

Together with~\zcref{lem:prismsubdivdifferentlengths} this yields the following corollary.

\begin{corollary}\label{cor:hightwcyclepacking}
    Let $k \in \mathbb{N}^+$. Every graph $G$ with $\mathsf{tw}(G) \in \Omega(k^5\mathsf{polylog}(k))$, contains $k$ vertex-disjoint cycles of distinct lengths.
\end{corollary}
\begin{proof}
    By \zcref{thm:largetreewidthpartition}, $G$ contains vertex-disjoint subgraphs $G_1, \dots, G_k$ such that $\mathsf{tw}(G_i) \in \Omega(k^2)$. By \zcref{thm:prisms}, each $G_i$ contains a subdivision of the $6k$-prism. By \zcref{lem:prismsubdivdifferentlengths}, $|\Lambda(G_i)| \geq k$ for each $G_i$. Thus, we can choose one cycle from each $G_i$ to obtain $k$ vertex-disjoint cycles of distinct lengths.
\end{proof}

We are now ready to prove a non-algorithmic version of~\zcref{thm:main}.

\begin{theorem} \label{thm:main2prime}
    For every $k \in \mathbb{N}$, every graph $G$ contains $k$ vertex-disjoint cycles of different lengths, or there exists a set $X \subseteq V(G)$ with $|X| \in \mathcal{O}(k^6\mathsf{polylog}(k))$ such that $|\Lambda(G-X)| \leq k-1$. 
\end{theorem}
\begin{proof}
    By \zcref{cor:hightwcyclepacking}, we may assume $\mathsf{tw}(G) \in \mathcal{O}(k^5\mathsf{polylog}(k))$. An application of \zcref{lem:lowtwdistinctlengthep} then yields $k$ vertex-disjoint cycles of different lengths or the set $X \subseteq V(G)$ with the desired properties. 
\end{proof}

An algorithmic proof of this theorem is provided in \zcref{sec:algoproof}.

\section{Lower bounds} \label{sec:lowerbounds}
In order to quantify the tightness of our bounds, we introduce the following definition.  
Let $f$ and $g$ be functions $\mathbb{N} \to \mathbb{R}$.  We say that $(f,g)$ is a \emph{binding pair} if for all $k \in \mathbb{N}$, every graph $G$ contains $k$ vertex-disjoint cycles of distinct lengths, or a set of at most $f(k)$ vertices $X$ such that $|\Lambda(G-X)| \leq g(k)$.  We can restate our main theorem in terms of binding pairs as follows. 

\begin{reptheorem}{thm:main}
Let $g(k)=k-1$. Then there exists a function $f(k) \in \mathcal{O}(k^6\mathsf{polylog}(k))$ such that $(f,g)$ is a binding pair.
\end{reptheorem}

We now show that $g(k)=k-1$ is optimal for binding pairs $(f,g)$ in the following strong sense.  

\begin{theorem} 
\label{lowerboundg}
If $(f,g)$ is a binding pair, then $g(k) \geq k-1$ for all $k \in \mathbb{N}$.  
\end{theorem}

\begin{proof}
Let $G$ be the disjoint union of $f(k)+1$ copies of $K_{k+1}$.  Then $G$ does not contain $k$ vertex-disjoint cycles with distinct lengths, but $|\Lambda(G-X)| \geq k-1$ for all $X \subseteq V(G)$ with $|X| \leq f(k)$.    
\end{proof}

On the other hand, we now show that $f(k) \in \mathcal{O}(k \log k)$ is best possible for binding pairs $(f,g)$ in the following strong sense.  

\begin{theorem}
\label{lowerboundf}
If $(f,g)$ is a binding pair, then $f \in \Omega(k \log k)$.
\end{theorem}

\begin{proof}
By the classic lower bound construction in~\cite{ErdosP1965Independent}, there exist graphs $G$ without $k$ vertex-disjoint cycles such that $|X| \in \Omega(k \log k)$ for all $X \subseteq V(G)$ such that $G-X$ is a forest. We construct a graph $G'$ from $G$ by replacing each $e=uv \in E(G)$ by a path $P_e\coloneqq w_1 \dots w_{2g(k+1)+1}$ and extra edges $w_1w_3, w_3w_5, \dots, w_{2g(k+1)-1}w_{2g(k+1)+1}$, where $w_1=u$ and $w_{2g(k+1)+1}=v$. Let $W$ be the set of new vertices.  Thus, $|W|=(2g(k+1)-1)|E(G)|$ and $V(G')=W \cup V(G)$.   Since $G$ does not contain $k$ vertex-disjoint cycles, we conclude that $G'$ does not contain $k+1$ vertex-disjoint cycles with distinct lengths.  Suppose that $X \subseteq V(G')$ is such that $|\Lambda(G'-X)| \leq g(k+1)$.  We may clearly assume that $X \subseteq V(G)$.  We claim that $G-X$ is a forest.  If not, then $G-X$ contains a cycle $C$.  Since $X \cap W=\emptyset$, we conclude that $|\Lambda(G'-X)| \geq g(k+1)|V(C)| > g(k+1)$, which is a contradiction.  Thus, $G-X$ is a forest, and so $|X| \in \Omega(k \log k)$, as required.
\end{proof}
% \CMin{I don't understand this proof. If we subdivide each edge the same number of times, then the length of every cycle is multiplied by some constant. This won't change the size of the cycle spectrum (for any subgraph).}

% \THin{We also add extra edges, so it is a path of 'triangles'. The indexing of the extra edges was off, now fixed.  This allows us to increase the size of the cycle spectrum by $g(k+1)$ for each original edge.}

\section{Coloured faces of embedded graphs}\label{sec:faces}
For our second main theorem, our goal is to pack or cover faces with distinct lengths in embedded graphs.  We, in fact, prove a more general theorem for coloured faces, where the colour of a face is not necessarily equal to its length.  We begin by reviewing some basic definitions from topological graph theory.  We refer the reader to~\cite{MoharT2001Graphs} for more details.
The \emph{Euler genus} of a surface obtained from the sphere by $h$ handles and $c$ crosscaps is $2h+c$.
% For each $h \in \mathbb{N}$, we let $\mathbb{T}_h$ denote the orientable surface with $h$ handles and we let $\mathbb{N}_h$ denote the nonorientable surface with $h$ crosscaps. 
% The \emph{Euler genus} of $\mathbb{T}_h$ and $\mathbb{N}_h$ is $2h$ and $h$, respectively.  
% For this purpose, we must first introduce the appropriate definitions to discuss embedded graph.
% Since our proofs permit us to speak about something more general than the lengths of cycles, we will discuss colourings of faces instead.
% We note that colourings in particular let us encode face-lengths, but are more general.

\subsection{Embedded graphs, surfaces, and colouring}\label{subsec:embeddedgraphs}
% For proper definitions of \emph{surfaces}, \emph{disks}, the \emph{Euler genus of a surface}, and \emph{boundaries} of surfaces we point to \cite{MoharT2001Graphs}.
% We will primarily be concerned with embeddings.

% \paragraph{Surfaces.}
% A \emph{surface} $\Sigma$ is a compact 2-dimensional manifold with or without boundaries.
% The surfaces we consider are generally assumed to be connected.
% Given a pair $(\mathsf{h},\mathsf{c}) \in \mathbb{N} \times [0,2]$, we let $\Sigma^{\mathsf{(\mathsf{h},\mathsf{c})}}$ be the surface without boundary created from the sphere by adding $\mathsf{h}$ handles and $\mathsf{c}$ crosscaps.
% Dyck's theorem \cite{Dyck1888Beitraege,FrancisW1999Conways} tells us two crosscaps are equivalent to a handle in the presence of a third crosscap, and thus our notation is sufficient to capture all surfaces without boundary.

% Given a surface $\Sigma$, we add an \emph{open}, respectively \emph{closed hole} to $\Sigma$ by removing a closed, respectively open disk from $\Sigma$.
% If $\Sigma$ is a surface (with holes), we let $\overline{\Sigma}$ be the surface resulting from glueing a closed, respectively open disk onto each open, respectively closed hole of $\Sigma$.
% We let the \emph{Euler genus} of $\Sigma$ be $2\mathsf{h} + \mathsf{c}$, where $\Sigma^{\mathsf{(\mathsf{h},\mathsf{c})}}$ is the surface to which $\overline{\Sigma}$ is isomorphic.
% Given a surface $\Sigma$, we denote by $\mathsf{bd}(\Sigma)$ the \emph{boundary} of said surface.

\paragraph{Embeddings}
Given a graph $G$, we say that $G$ has an \emph{embedding} $\psi$ into a surface $\Sigma$, if $\psi$ is a function with domain $V(G) \cup E(G)$ such that
\begin{enumerate}
    \item $\psi(v)$ is a point of $\Sigma$ for each $v \in V(G)$,
    \item $\psi(u) \neq \psi(v)$ for all distinct $u,v \in V(G)$,
    \item $\psi(uv)$ is a simple curve in $\Sigma$ with the ends $\psi(u),\psi(v)$ for all $uv \in E(G)$,
    \item $\psi(uv) \cap \psi(wx) \subseteq \{ \psi(u), \psi(v) \}$, for all distinct $uv, wx \in E(G)$, and
    \item $\psi(uv) \cap \psi(w) = \emptyset$ for all $uv \in E(G)$ and $w \in V(G) \setminus \{ u,v \}$.
\end{enumerate}
We let $\psi(G)$ be the union of all points and curves in its image and call the connected components of $\Sigma \setminus \psi(G)$ the \emph{faces} of $\psi$, and we let $\mathcal{F}(\psi)$ be the set of faces of $\psi$.
If every face of $\psi$ is homeomorphic to an open disk, we say that $\psi$ is a \emph{2-cell embedding}.  If, in addition, the boundary of each face is a simple closed curve, then $\psi$ is a \emph{closed 2-cell embedding}.

A particular advantage of closed 2-cell embeddings is that they allow us to define simple dual graphs of embedded graphs.
Let $G$ be a graph with a closed 2-cell embedding $\psi$ in some surface.
Then the \emph{dual graph} $G^*$ of $G$ is a graph with vertex set $\mathcal{F}(\psi)$ and the edge $FF' \in E(G^*)$ if the facial cycles of the faces $F$ and $F'$ of $\psi$ intersect in an edge.
Under our assumptions, $G^*$ is a simple graph~\cite[Proposition 5.5.11.]{MoharT2001Graphs}.  
We generally drop the reference to $\psi$, as in our applications, the embedding is clear from the context.  By a slight abuse of notation, we let $V(F)$ be the vertices of $G$ which are on the boundary of $F$.

\paragraph{Colouring graphs on surfaces.}
Given an integer $k$ and a set $S$, we say that a function $c \colon S \to [k]$ is a \emph{$k$-colouring of $S$}.
If we do not want to specify $k$, we simply call $c$ a \emph{colouring of $S$}.
Note that we have no further requirements on $c$.
If $c$ is a colouring of the vertex set $V(G)$ of a graph $G$, we call $c$ \emph{proper} if for all edges $uv \in E(G)$ we have $c(u) \neq c(v)$.
A graph $G$ is \emph{$k$-colourable} if there exists a proper $k$-colouring of $V(G)$ and we denote by $\chi(G)$ the minimum number $\ell$ such that $G$ has a proper $\ell$-colouring.

We also need the following classic result on colouring embedded graphs, which is a combination of the Four Colour Theorem \cite{AppelH1976Every,AppelH1976Existence,AppelH1977Every,AppelHK1977Every} and Heawood's classic result on colouring graphs embedded on surfaces other than the sphere~\cite{Heawood1890MapColour}.

\begin{theorem}\label{prop:chromaticNumberGenusBound}
    Let $\Sigma$ be a surface of Euler genus $g$ and let $G$ be a graph embedded in $\Sigma$.
    Then
    \[ \chi(G) \leq \lfloor \frac{7 + \sqrt{1+24g}}{2} \rfloor . \]
\end{theorem}

\subsection{Packing and covering coloured faces}
We begin with a warm-up exercise.
Since deleting a vertex on a face creates a new face not found in the original graph, we use the \textsl{packing} and \textsl{covering} framing for the Erd\H{o}s-P\'osa property that we presented in the introduction instead of deleting vertices.

\begin{theorem}\label{thm:edgedisjointfaces}
    Let $k \in \mathbb{N}^+$ and $g \in \mathbb{N}$.
    Let $G$ be a graph with a closed 2-cell embedding $\psi$ on a surface $\Sigma$ of Euler genus $g$ and let $c$ be a colouring of $\mathcal{F}(\psi)$.

    Then there exist $k$ faces $F_1, \ldots, F_k \in \mathcal{F}(\psi)$ given pairwise distinct colours by $c$ and such that their facial cycles are pairwise edge-disjoint, or $|\{ c(F) \mid F \in \mathcal{F}(\psi)\}| < \lfloor \frac{7 + \sqrt{1+24g}}{2} \rfloor k$. 
\end{theorem}
\begin{proof}
    Suppose there exists at least $\lfloor \frac{7 + \sqrt{1+24g}}{2} \rfloor k$ different colours.
    Then we can choose one face from each colour class, and consider the dual graph induced on these faces.
    This graph again embeds in $\Sigma$, and so contains an independent set on $k$ vertices by~\zcref{prop:chromaticNumberGenusBound}.
    This independent set gives the desired edge-disjoint collection of $k$ faces.
\end{proof}

To see how this is related to our previous questions, we can define the colour of a face to be its length.  
\zcref{thm:edgedisjointfaces} now tells us that we can pack $k$ pairwise edge-disjoint facial cycles of different lengths, or the set of face lengths of $G$ is already small.

The proof of \zcref{thm:edgedisjointfaces} does not generalise directly to vertex-disjoint facial cycles, which tells us that the dual graph misses some face intersections, namely those that consist of a single vertex.
We thus define a notion of distance for faces of embedded graphs which captures this.

Let $G$ be a graph with a 2-cell embedding $\psi$ on a surface $\Sigma$ of Euler genus $g$.
Further, let $G_\mathcal{F}$ be the graph with vertex set $\mathcal{F}(\psi)$ such that two faces $F, F'$ are adjacent if and only if $V(F) \cap V(F') \neq \emptyset$. We call $G_\mathcal{F}$ the \emph{face-incidence graph of $G$ and $\psi$}.
We let the \emph{distance in $\psi$} between two faces $F,F' \in \mathcal{F}(\psi)$ be the distance between $F$ and $F'$ in $G_\mathcal{F}$.

We further need to define the \emph{radial graph} $R(G,\psi)$ of a graph $G$ with a 2-cell embedding $\psi$ on some surface $\Sigma$.
We construct $R(G,\psi)$ from $G$ by adding a vertex $v_F$ for each $F \in \mathcal{F}(\psi)$ and then adding the edges $v_Fu$ for all $F \in \mathcal{F}(\psi)$ and all $u \in V(F)$. We note that $\psi$ can be extended to an embedding of $R(G,\psi)$ in $\Sigma$ by placing $v_F$ into $F$ and drawing the additional edges incident to $v_F$ inside $F$ in the natural fashion. Note that if $F$ and $F'$ are at distance $d$ in the face-incidence graph $G_\mathcal{F}$, then $v_F$ and $v_{F'}$ are at distance at most $2d$ in $R(G,\psi)$.

For any given graph $G$, integer $r \in \mathbb{N}$, and $v \in V(G)$, we let $N^r_G(v)$ be the \emph{$r$-neighbourhood of $v$ (in $G$)}, which is the set of all vertices in $V(G)$ that lie at distance at most $r$ from $v$ in $G$.
Note that for all $r \in \mathbb{N}$ and $v \in V(G)$, we thus have $v \in N^r_G(v)$. For $X \subseteq V(G)$, we let $N^r_G(X)$ denote $\bigcup_{x \in X} N^r_G(x)$.

We say that a graph $G$ has \emph{local treewidth} at most $f$ if $f: \mathbb{N} \rightarrow \mathbb{N}$ is a function such that for any vertex $v \in V(G)$, we have $\mathsf{tw}(G[N^d(v)]) \leq f(d)$. Eppstein showed that a minor-closed class of graphs has bounded local treewidth if and only if it does not include all apex-graphs \cite{Eppstein2000Diameter} (see also \cite{HendreyW2025Polynomial} for recent improvements giving polynomial bounds). We use in particular that graphs of bounded genus have bounded local treewidth, as given below.
% For this purpose, we extract a result from a recent improvement of a specific case of the so-called \emph{apex minor theorem} in \cite{HendreyW2025Polynomial}.

% \begin{theorem}[Hendrey and Wood \cite{HendreyW2025Polynomial} (see Theorem 7)]\label{lem:bndgenusbndradiushasbndtw}
%     For every graph $G$ with an embedding on a surface $\Sigma$ of Euler genus $g$ and radius at most $r$, we have $\mathsf{tw}(G) \in \mathcal{O}^*(r^9(1+\sqrt{g})^9)$.
% \end{theorem}

\begin{theorem}[Eppstein~{\cite[Theorem~2]{Eppstein2000Diameter}}]\label{lem:bndgenusbndradiushasbndtw}
    For all $r,g \in \mathbb{N}$, every graph $G$ of radius at most $r$ with an embedding on a surface $\Sigma$ of Euler genus $g$ has treewidth $\mathcal{O}(r(g+1))$.
\end{theorem}

We derive \zcref{thm:facesatdistance} from a more general fact about packing far apart vertices of distinct colours, see \zcref{thm:colouredverticesatdistanceEP}. Before we prove this, we first state an analogue of the algorithmic version of \zcref{lem:lowtwdistinctlengthep}.
A proof of this statement is provided in \zcref{sec:facesapp}.

\begin{lemma}\label{lem:bndtwcolouredverticesatdistanceEP}
    Let $k,d \in \mathbb{N}^+$. Let $G$ be a graph with $\mathsf{tw}(G) \leq \ell$, and let $c$ be a colouring of $V(G)$. Then there exist $k$ vertices that are given pairwise distinct colours by $c$ and are pairwise at distance at least $d$, or there exists a set $X \subseteq V(G)$ with $|X| \leq (k-1)(\ell + 1)$ such that $|\{c(v) \mid v \in V(G - N^{\lfloor d/2\rfloor}(X))\}| \leq k-1$.

    In particular, there exists an algorithm running in time $2^{\mathcal{O}(\ell^2)}\mathsf{poly}(k,d,\ell)|V(G)|^4$ that finds either the packing of $k$ vertices or the set $X$ with the properties promised above.
\end{lemma}

We now use the above lemma to get a similar result for graphs that only locally have low treewidth.

% \CMin{The below theorem might be true more generally for $H$-minor free graphs (instead of just apex-minor free graphs). I can show how to get it for the local structure: after you delete a bounded number of apex vertices, you have bounded local treewidth. Thus, you can apply the below result after colouring all vertices within distance $\lfloor d/2\rfloor$ from apices by a single special colour, and add these apices to any hitting set. However, I can't figure out how to deal with a tree-decomposition of such things.}
% \MGin{Hmm... I see what you are doing. But I think we should wrap this project up. Maybe we can pose this as an open problem in the last section? We could explicitly mention that it is fairly clear how to do this 'locally'. Do you see some way to apply this to get the sort of Erdös-Posa property we want for something in $H$-minor-free graphs? The theorem you propose seems kind of unsuited for that, since the concept of 'faces' as such disappears once we go properly beyond bounded genus.}
% \CMin{I wrote something at the bottom of the last section about it.}

\begin{theorem}\label{thm:colouredverticesatdistanceEP}
    Let $k,d \in \mathbb{N}^+$.
    Let $G$ be a graph with local treewidth at most $f$ and let $c$ be a colouring of $V(G)$.
    Then there exist $k$ vertices that are given pairwise distinct colours by $c$ and pairwise lie at distance at least $d$, or there exists a set $X \subseteq V(G)$ of order at most $(k^2-k)(f(2d) + 1)$ such that $|\{c(v) \mid v \in V(G - N^{\lfloor d/2\rfloor}(X))\}| \leq k^2 - 1$.

    Moreover, there exists an algorithm running in time $2^{\mathcal{O}(f(2d)^2)}\mathsf{poly}(k,d)|V(G)|^4$ that finds either the packing of $k$ vertices or the set $X$ with the properties promised above.
\end{theorem}
\begin{proof}
    We first greedily pack vertices of pairwise distance at least $d$ and of distinct colours. If we do not find the desired collection of vertices, then we have a set of at most $k-1$ vertices $v_1, \dots, v_b$ such that every remaining vertex has a colour in $\{c(v_i) \mid i \in [b]\}$, or is at distance less than $d$ from $\{v_1, \dots, v_b\}$.

    We now fix $i \in [b]$ and restrict to $H_i \coloneq G[N^{2d}(v_i)]$. We define a new colouring $c_i$ of $V(H_i)$ as $c_i(v) = c(v)$ for all $v \in N^d(v_i)$, and every other vertex is assigned some fixed new colour. By assumption, $\mathsf{tw}(H_i) \leq f(2d)$. We apply \zcref{lem:bndtwcolouredverticesatdistanceEP} to find a set of $k+1$ vertices in $H_i$ which are pairwise distance at least $d$ and given distinct colours by $c_i$, or we obtain a set $X_i$ of size at most $k(f(2d) + 1)$ such that $|\{c_i(v) \mid v \in V(H_i - N^{\lfloor d/2\rfloor}(X_i)\}| \leq k$.
    If, for some $i \in [b]$, we find a set of $k+1$ vertices, then we note that at least $k$ of them must be in $N^d(v_i)$. For any pair of such vertices, if their distance in $G$ were less than $d$, then a shortest path between them would lie in $H_i$, contradicting their distance in $H_i$. Hence, we obtain the desired set of $k$ vertices in $G$. Therefore, we may assume that we obtain the set $X_i$ for each $i \in [b]$.

    We let our final hitting set be $X = \bigcup_{i \in [b]} X_i$, which has size at most $(k^2-k)(f(2d)+1)$. Every vertex of distance more than $\lfloor d/2\rfloor$ from $X$ has colours in $\{c(v_i) \mid i \in [b]\} \cup \bigcup_{i \in [b]} \{c(v) \mid v \in V(H_i - N^{\lfloor d/2\rfloor}(X_i))\}$, which has size at most $(k-1) + k(k-1) = k^2 - 1$. To see this, note that every vertex is either of distance at least $d$ from $\{v_1, \dots, v_b\}$ and hence has colour in $\{c(v_i) \mid i \in [b]\}$, or it is in $N^d(v_i)$ for some $i$ and thus $c_i(v) = c(v)$.
\end{proof}

We note that~\zcref{thm:colouredverticesatdistanceEP} does not generalise to the class of all graphs. Consider a subdivided clique so that every pair of high degree vertices lies at distance $d-1$. Colour every original vertex by a unique colour, and every subdividing vertex by a single colour. This graph contains no 3 vertices of pairwise distance $d$ and of distinct colours, yet $|X|$ and the number of remaining vertices must grow with the size of the clique.

We also note that the following approximate converse of \zcref{thm:colouredverticesatdistanceEP} holds.

\begin{observation}
    Let $G$ be a graph and $c$ be a colouring of $V(G)$. Suppose $X$ is a subset of $V(G)$ of order at most $k$ such that there are at most $q$ colours among the vertices in $G - N^{d}(X)$. Then $G$ has no packing of $k + q + 1$ vertices each of distinct colours and with pairwise distance at least $2d+1$.
\end{observation}

Hence, if $d$ is odd, the second outcome of \zcref{thm:colouredverticesatdistanceEP} implies that there is no packing of $\mathcal{O}(k^2f(2d))$ vertices of distinct colours and with pairwise distance $d$.

We can now obtain \zcref{thm:facesatdistance}, which we restate here, with little additional effort.
% \THin{This statement does not match the statement in the Abstract and Introduction. Also, should the bound be $(k+1)^2-1$ instead of $k(k+1)$?}
% \MGin{I fixed the $\mathcal{O}(k^2dg)$ thing in the introduction. So that now fits the statement here. You were also right about the colour-bound having to be $(k+1)^2-1$. I fixed this in the proof here and both the statement here and in the intro. Then I adjusted the abstract accordingly.}

\begin{reptheorem}{thm:facesatdistance}
    Let $k,d \in \mathbb{N}^+$ and $g \in \mathbb{N}$.
    Let $G$ be a graph with a 2-cell embedding $\psi$ on a surface $\Sigma$ of Euler genus $g$, let $c$ be a colouring of the faces $\mathcal{F}(\psi)$ of $\psi$, and let $R(G,\psi)$ be the radial graph of $(G, \psi)$.

    Then there exist $k$ faces $F_1, \ldots , F_k \in \mathcal{F}(\psi)$ that are given pairwise distinct colours by $c$ and are pairwise at distance at least $d$ in $\psi$, or there exists a set $X \subseteq V(G)$ of order at most $\mathcal{O}(k^2d(g+1))$ such that $|\{ c(F) \mid F \in \mathcal{F}(\psi) \text{ and } V(F) \cap \bigcup_{x \in X} N^d_{R(G,\psi)}(x) = \emptyset \}| \leq k(k+2)$.

    Moreover, there exists an algorithm running in time $2^{\mathcal{O}(d^2(g+1)^2)}\mathsf{poly}(k)|V(G)|^4$ that finds either the faces $F_1, \ldots, F_k$ or the set $X$ with the properties promised above.
\end{reptheorem}
\begin{proof}
    We apply \zcref{thm:colouredverticesatdistanceEP} to the radial graph $R(G,\psi)$ and the colouring $c'$ where for each vertex $v_F$ coming from the face $F \in \mathcal{F}(\psi)$ we set $c'(v_F) = c(F)$, and for every other vertex in $V(G)$ we give them all the same new colour. Because $R(G,\psi)$ embeds on a surface of genus $g$, \zcref{lem:bndgenusbndradiushasbndtw} shows that $R(G,\psi)$ has local treewidth at most $\mathcal{O}(d(g+1))$.
    
    If we have a packing of $k+1$ vertices in $R(G,\psi)$ each given distinct colours by $c'$ and pairwise at distance at least $2d$, then at least $k$ of them correspond to faces. For any two faces $F, F'$, the distance between $v_F$ and $v_{F'}$ in $R(G,\psi)$ is at most twice the distance between $F$ and $F'$ in $\psi$. Hence, we obtain the desired packing of $k$ faces.

    If instead, \zcref{thm:colouredverticesatdistanceEP} returns a set $X$ of order in $\mathcal{O}(k^2d(g+1))$, such that there are at most $(k+1)^2-1$ colours outside $N^d_{R(G,\psi)}(X)$, then we modify $X$ to get our desired set as follows. For every vertex $v_F \in X$ corresponding to some face $F$, we replace $v_F$ by an arbitrary vertex in $V(F)$. Because this vertex is adjacent to $v_F$, we see that $N^{d+1}_{R(G,\psi)}(X)$ now covers all but $(k+1)^2 - 1$ colours. However if a vertex $v_F$ is in $N^{d+1}_{R(G,\psi)}(X)$, then $V(F) \cap N^d_{R(G,\psi)}(X) \neq \emptyset$. Hence, this modified $X$ is the desired set.

    Clearly, the algorithm given by \zcref{thm:colouredverticesatdistanceEP} can be used to find the desired faces or set of vertices $X$ as desired.
\end{proof}

Note that applying the above with $d=1$, meaning that we want vertex-disjoint faces, yields a hitting set of balls of radius 1 instead of single vertices. However, we can easily rectify this by a slight modification of the proof. If $X$ is the set returned by \zcref{thm:colouredverticesatdistanceEP}, then for every vertex $x \in X$, if $x = v_F$ for some face $F$, then $N^1(x)$ only covers that singular face vertex, so we may replace $x$ by an arbitrary vertex of $V(F)$. If instead $x \in V(G)$, then $x$ only covers $v_F$ for $F$ such that $x \in V(F)$. Hence, this single vertex is enough to hit all the corresponding faces. This gives the following corollary.

\begin{corollary}
    Let $k \in \mathbb{N}^+$ and $g \in \mathbb{N}$.
    Let $G$ be a graph with a 2-cell embedding $\psi$ on a surface $\Sigma$ of Euler genus $g$ and let $c$ be a colouring of the faces $\mathcal{F}(\psi)$ of $\psi$.

    Then there exist $k$ faces $F_1, \ldots, F_k \in \mathcal{F}(\psi)$ given pairwise distinct colours by $c$ and such that their facial cycles are pairwise disjoint, or there exists a set $X \subseteq V(G)$ of order at most $\mathcal{O}(k^2(g+1))$ such that $|\{ c(F) \mid F \in \mathcal{F} \text{ and } V(F) \cap X = \emptyset \}| \leq k(k+2)$.

    Moreover, there exists an algorithm running in time $2^{\mathcal{O}((g+1)^2)}\mathsf{poly}(k)|V(G)|^4$ that finds either the faces $F_1, \ldots, F_k$ or the set $X$ with the properties promised above.
\end{corollary}

Finally, we remark that in the proof of \zcref{thm:colouredverticesatdistanceEP}, we use that the neighbourhood around a vertex has bounded treewidth. But the neighbourhood around a facial cycle also has bounded treewidth, as for any facial cycle $F$, we have $N_G^d(V(F)) \subseteq N^d_{R(G,\psi)}(v_F)$. Thus, by following the same proof as for~\zcref{thm:colouredverticesatdistanceEP}, we may also obtain the following result.

\begin{theorem}
    Let $k,d \in \mathbb{N}^+$ and $g \in \mathbb{N}$.
    Let $G$ be a graph with a 2-cell embedding $\psi$ on a surface~$\Sigma$ of Euler genus $g$, and let $c$ be a colouring of the faces $\mathcal{F}(\psi)$ of $\psi$.

    Then there exist $k$ faces $F_1, \ldots, F_k \in \mathcal{F}(\psi)$ that are given pairwise distinct colours by $c$ and their facial cycles are pairwise at distance at least $d$ in $G$, or there exists a set $X \subseteq V(G)$ of order at most $\mathcal{O}(k^2d(g+1))$ such that $|\{ c(F) \mid F \in \mathcal{F}(\psi) \text{ and } V(F) \cap N^d_G(X) = \emptyset \}| \leq \mathcal{O}(k^2)$.
\end{theorem}

\section{Subdivided ladders with few cycle lengths} \label{sec:bounds}
Recall that the \emph{$k$-ladder} is the Cartesian product of $K_2$ and a path with $k$ vertices.  Note that a $k$-ladder is a subgraph of a $k$-prism.  However, unlike prisms, a $k$-ladder is a subgraph of an $\ell$-ladder for every $\ell \geq k$.  Therefore, a natural approach to improve the bound in our main theorem is to use $k$-ladders instead of $k$-prisms in \zcref{cor:hightwcyclepacking}. 

\begin{question}\label{con:manydistinctcyclesinladders2}
    Does every subdivision of the $k$-ladder contain $\Omega(k)$ (not necessarily vertex-disjoint) cycles with distinct lengths?
\end{question}

% \MSin{I agree, but do we want to give a one-sentence explanation here? (I don't insist.)}
% \THin{I added a one-sentence explanation. I think the explanation fits better below~\zcref{con:con:manydistinctcyclesinladders1}}
Observe that the Erd\H{o}s-Szekeres theorem implies a bound of $\Omega(\sqrt{k})$ in~\zcref{con:manydistinctcyclesinladders2}. 
By partitioning a $k^2$-ladder into $k$ disjoint $k$-ladders, we also note that a positive answer to \zcref{con:manydistinctcyclesinladders2} implies a positive answer to the following.

\begin{question} \label{con:con:manydistinctcyclesinladders1}
    Does every subdivision of the $k^2$-ladder contain $\Omega(k)$ vertex-disjoint cycles with distinct lengths?  
\end{question}

Because every graph with treewidth $\Omega(k^4)$ contains a subdivided $k^2$-ladder, a positive answer to \zcref{con:con:manydistinctcyclesinladders1} would improve the bound on $|X|$ in~\zcref{thm:main} from $\mathcal{O}(k^6\mathsf{polylog}(k))$ to $\mathcal{O}(k^5)$. Unfortunately, we now show that both~\zcref{con:con:manydistinctcyclesinladders1,con:manydistinctcyclesinladders2} have negative answers.
The counterexample utilises the following result in additive combinatorics.
For a~finite set $A \subseteq \mathbb{Z}$, define the \emph{sum set} $A + A \coloneqq \{a + b \mid a, b \in A\}$ and the \emph{difference set} $A - A \coloneqq \{a - b \mid a, b \in A\}$.
A~construction by Hennecart, Robert, and Yudin \cite{HennecartRY99} yields sets $A$ that have a significantly larger difference set than sum set.
In the following statements, we fix a~constant $C \coloneqq \frac{\log 2}{\log (1 + \sqrt{2})} < 0.7865$.

\begin{theorem}[Hennecart, Robert, and Yudin \cite{HennecartRY99}]
    \label{thm:small-sumsets}
    For every $n \in \mathbb{N}$ and $\alpha > C$, there exists a~set $A \subseteq \mathbb{Z}$ such that $|A| \geq n$ and $|A + A| \leq |A - A|^\alpha$.
\end{theorem}

We use~\zcref{thm:small-sumsets} to create a~finite point set $P \subseteq \mathbb{Z}^2$ in the two-dimensional plane with specific properties.
We shall say that $P$ is a~\emph{constellation} if all points in $P$ have pairwise different $x$-coordinates, namely, $x_1 \neq x_2$ for all distinct $(x_1, y_1), (x_2, y_2) \in P$.
For two points $p_1 = (x_1, y_1)$, $p_2 = (x_2, y_2)$, define $d(p_1, p_2) = y_1 + y_2 + |x_1 - x_2|$.
The \emph{variety} of a~constellation $P$ is the number of different values of $d(p_1, p_2)$ for distinct $p_1, p_2 \in P$, i.e.,
\[
    |\{
        d(p_1, p_2) \mid p_1, p_2 \in P,\, p_1 \neq p_2
    \}|.
\]
We now show that the variety of a~constellation can be significantly smaller than its cardinality.

\begin{lemma}
    \label{lem:constellation-with-small-variety}
    For every $n \in \mathbb{N}$ and $\alpha > C$, there exists a~constellation $P \subseteq \mathbb{Z}^2$ such that $|P| \geq n$ and the variety of $P$ is at most $|P|^\alpha$.
\end{lemma}
\begin{proof}
    By~\zcref{thm:small-sumsets}, there exists a~finite $A \subseteq \mathbb{Z}$ with $|A| \geq n$ and $|A + A| \leq |A - A|^\alpha$.
    We construct the point set $Q \subseteq \mathbb{Z}^2$:
    \[
        Q \coloneqq \{ (a - b, a + b) \mid a, b \in A \}.
    \]
    Note that a~point $(x, y) = (a - b, a + b)$ is formed by the intersection of the lines $x + y = 2a$ and $y - x = 2b$ in the Euclidean plane.
    Hence, $Q$ is a~``grid'' consisting of $n^2$ points, formed by the intersections of all the lines of the form $\{x + y = 2a \mid a \in A\}$ with all the lines of the form $\{y - x = 2b \mid b \in A\}$.
    We choose $P$ to be a maximal subset of $Q$ with pairwise distinct $x$-coordinates.

    We now claim that $P$ satisfies the conditions of the statement of the lemma.
    Observe that $|P| = |A - A| \geq |A| \geq n$ since the set of $x$-coordinates of points from $P$ is precisely $A - A$.
    Thus, it remains to estimate the variety of $P$.
    Choose any pair of distinct points $p_1 = (x_1, y_1), p_2 = (x_2, y_2) \in P$; and let $a_i, b_i \in A$ for $i \in [2]$ be such that $p_i$ is the intersection of the lines $x + y = 2a_i$ and $y - x = 2b_i$ in the plane.
    Without loss of generality, assume that $x_1 < x_2$.
    Then
    \[ d(p_1, p_2) = (y_1 - x_1) + (x_2 + y_2) = 2(b_1 + a_2) \in 2(A + A), \]
    where $2(A + A) = \{2t \mid t \in A + A\}$.
    Therefore the variety of $P$ is upper-bounded by $|2(A + A)| = |A + A| \leq |A - A|^\alpha = |P|^\alpha$, as required.
\end{proof}

To prove~\zcref{lem:ladder-subdivision-with-few-cycles}, we now translate the constellation of~\zcref{lem:constellation-with-small-variety} into a~subdivision of a~ladder with few cycle lengths.

\begin{reptheorem}{lem:ladder-subdivision-with-few-cycles}
    For every $n \in \mathbb{N}$ and $\alpha > C$, there exists $k \geq n$ and a~subdivision $G$ of the $k$-ladder such that $|\Lambda(G)| \leq k^\alpha$.
\end{reptheorem}
\begin{proof}
    Pick $P = \{p_1, \ldots, p_k\} \subseteq \mathbb{Z}^2$ as a~constellation of cardinality $k \geq n$ and variety at most $k^\alpha$, as guaranteed by~\zcref{lem:constellation-with-small-variety}.
    Let $p_i = (x_i, y_i)$ for each $i \in [k]$, assume that $x_1 < \ldots < x_k$, and let $s < \min \{y_1, \ldots, y_k\}$ be an~integer.
    Let also $u_1, \ldots, u_k, v_1, \ldots, v_k$ be the vertices of the $k$-ladder, with edges of the form $u_i u_{i+1}$, $v_i v_{i+1}$ for $i \in [k-1]$, and $u_i v_i$ for $i \in [k]$.
    In order to construct $G$, it is enough to assign a~positive integer weight to each edge of the ladder and define $G$ by replacing each edge with a~path of length equal to the weight of the edge.
    In our construction:
    \begin{itemize}
        \item the edges $u_i u_{i+1}$ and $v_i v_{i+1}$ are assigned weights $x_{i+1} - x_i$, and
        \item the edge $u_i v_i$ is assigned weight $2(y_i - s)$.
    \end{itemize}
    It remains to analyse the lengths of the cycles of $G$.
    Let $1 \leq i < j \leq k$ and consider the cycle of $G$ that is a~subdivision of the cycle $u_i \ldots u_j v_j \ldots v_i$ of the $k$-ladder.
    Its length is precisely
    \[
        2(x_j - x_i) + 2(y_i - s) + 2(y_j - s) = 2(y_i + y_j + |x_i - x_j|) - 4s = 2d(p_i, p_j) - 4s.
    \]
    Hence $\Lambda(G) = \{2d(p_i, p_j) - 4s \mid 1 \leq i < j \leq k\}$, and so $\Lambda(G)$ is precisely the variety of $P$, which has size at most $k^\alpha$ by~\zcref{lem:constellation-with-small-variety}.
\end{proof}

This refutes~\zcref{con:manydistinctcyclesinladders2}.
It remains to refute~\zcref{con:con:manydistinctcyclesinladders1}.
This follows immediately from the following statement.

\begin{theorem} \label{thm:question2}
    Let $D \coloneqq \frac{C}{1 + C} < 0.4403$.
    For every $m \in \mathbb{N}$ and $\beta > D$, there exists $\ell \geq m$ and a~subdivision of the $\ell$-ladder containing no family of $\ell^\beta$ vertex-disjoint cycles of pairwise different lengths.
\end{theorem}
\begin{proof}
    Let $\alpha > C$ be such that $\frac{\alpha}{1 + \alpha} < \beta$.
    For $n \in \mathbb{N}$, let $k_n \geq n$ and $G_n$ be a~subdivision of the $k_n$-ladder such that $|\Lambda(G_n)| \leq k_n^\alpha$, as provided by~\zcref{lem:ladder-subdivision-with-few-cycles}.
    Also, let $q_n \coloneqq \left\lceil k_n^\alpha \right\rceil$.
    We construct the graph $H_n$ by taking $q_n$ disjoint copies of $G_n$ and connecting them arbitrarily using $2(q_n - 1)$ additional edges to construct a~subdivision of the $(k_n q_n)$-ladder.
    Let $\ell_n \coloneqq k_n q_n$.

    We now show that $H_n$ contains no family of $2k_n^\alpha$ vertex-disjoint cycles of pairwise different lengths.
    Indeed, in any packing of cycles in $H_n$:
    \begin{itemize}
        \item at most $q_n - 1 < k_n^\alpha$ cycles contain vertices of more than one disjoint copy of $G_n$, since each such cycle utilises at least two edges connecting disjoint copies of $G_n$; and
        \item the disjoint copies of $G_n$ only contain cycles of lengths in $\Lambda(G_n)$, and $|\Lambda(G_n)| \leq k_n^\alpha$.
    \end{itemize}

    It remains to prove that for large enough $n \in \mathbb{N}$, we have $\ell_n \geq m$ and $2k_n^\alpha \leq \ell_n^\beta$.
    The former is obvious, and for the latter, we have
    \[
        \ell_n^\beta \geq 2 \ell_n^{\alpha / (1 + \alpha)} = 2(k_n q_n)^{\alpha / (1 + \alpha)} \geq 2(k_n \cdot k_n^\alpha)^{\alpha / (1 + \alpha)} = 2k_n^\alpha,
    \]
    where the first inequality holds for large enough $n$ (and so large enough $\ell_n$) by $\frac{\alpha}{1 + \alpha} < \beta$.
\end{proof}

We note that \zcref{thm:question2} also shows that there exist subdivided $k^2$-prisms with at most $\mathcal{O}(k^{0.8806})$ vertex-disjoint cycles of distinct lengths.

\medskip

\section{On possible extensions of our results}\label{sec:extensions}

\paragraph{Improved bounds.}  Given \zcref{lem:ladder-subdivision-with-few-cycles}, it appears quite challenging to improve our bounds without completely new ideas.  Nonetheless, the following conjecture could be true.  

\begin{conjecture} \label{wildconjecture}
    There exists a function $f(k) \in \mathcal{O}(k \log k)$ such that every graph $G$ contains $k$ vertex-disjoint cycles with distinct lengths, or a set $X$ of at most $f(k)$ vertices such that $|\Lambda(G-X)|
    \leq k-1$.
\end{conjecture}

By \zcref{lowerboundf, lowerboundg}, \zcref{wildconjecture} would be simultaneously optimal for both $|X|$ and $|\Lambda(G-X)|$.  Moreover, the transformation presented in the proof of \zcref{lowerboundf} shows that \zcref{wildconjecture} also implies the classic \EP{} theorem.  

\medskip

It is also natural to ask if our results can be extended to lengths which are not only distinct but also satisfy some additional nice properties.  This appears quite challenging given the following negative results. 

\paragraph{Consecutive lengths.} Does there exist a function $f(k)$ such that every graph $G$ contains $k$ vertex-disjoint cycles whose lengths are $k$ consecutive integers or a set $X$ of at most $f(k)$ vertices such that $\Lambda(G-X)$ does not contain $f(k)$ consecutive integers?  The following simple example shows that the answer is negative. Let $C_a$ denote the cycle of length $a$. Given $A \subseteq \mathbb{N}$, let $G_A=\bigcup_{a \in A} C_a$, where there exists a vertex $x \in V(G_A)$ such that $V(C_a) \cap V(C_b)=\{x\}$ for all distinct $a,b \in A$. Let $\mathcal{A}$ be a family of subsets of $\mathbb{N}$ such that each $A \in \mathcal{A}$ is a large subset of consecutive integers and $A \cup A'$ is not a set of consecutive integers for all distinct $A, A' \in \mathcal{A}$.  Let $G_{\mathcal{A}}$ be the disjoint union of $G_A$, over all $A \in \mathcal{A}$.  Note that $G_\mathcal{A}$ does not contain two vertex-disjoint cycles of consecutive lengths.  On the other hand if $|X| < |\mathcal{A}|$, then $\Lambda(G_{\mathcal{A}}-X)$ still contains long sequences of consecutive integers (since $X \cap V(G_A)=\emptyset$ for some $A \in \mathcal{A}$). This example actually shows that for all $t \in \mathbb{N}$, the $\nicefrac{1}{t}$-integral \EP{} property fails for consecutive cycle lengths. 

\paragraph{Arithmetic Progressions.}  What if instead of asking for consecutive lengths, we ask for lengths which form an arithmetic progression?  There are two natural ways to formalise this question, depending on whether we fix the common difference of the arithmetic progression or not.  To be precise, recall that a \emph{$d$-arithmetic progression} is a sequence of integers whose consecutive terms differ by $d$.  A sequence is an \emph{arithmetic progression} if it is a $d$-arithmetic progression for some $d$.  Fix $d \geq 1$. 
% \footnote{We exempt the trivial arithmetic progression in which all cycles would have the same length from the discussion here. In this case, the $\nicefrac{1}{t}$-integral \EP{} property holds for all positive integers $t$ for trivial reasons.} 
  Does there exist a function $f(k)$ such that every graph $G$ contains $k$ vertex-disjoint cycles whose lengths form a $d$-arithmetic progression or a set of at most $f(k)$ vertices $X$ such that $\Lambda(G-X)$ does not contain a $d$-arithmetic progression of size $f(k)$?  Alternatively, we can allow $d$ to vary. That is, does there exist a function $f(k)$ such that every graph $G$ contains $k$ vertex-disjoint cycles whose lengths form an arithmetic progression or a set of at most $f(k)$ vertices $X$ such that $\Lambda(G-X)$ does not contain any arithmetic progression of size $f(k)$?

A slight modification of the previous counterexample shows that the answer is negative for both questions. For the first question, this is straightforward by taking a $(d-1)$-subdivision of the previous counterexample. For the second question,  we say that a family $\mathcal{A}$ of subsets of $\mathbb{N}$ is
\emph{$(p,\ell)$-good} if each $A \in \mathcal{A}$ is a $p$-arithmetic progression of length $\ell$ and every arithmetic progression of length at least 3 in $\bigcup \mathcal{A}$ is a subset of $A$ for some $A \in \mathcal{A}$. We claim that for all $p,\ell \in \mathbb{N}$, there exist arbitrarily large $(p,\ell)$-good families. To see this, observe that if $\mathcal{A}$ is $(p,\ell)$-good, then $\mathcal{A} \cup \{A\}$ is $(p,\ell)$-good, where $A$ is a $p$-arithmetic progression of length $\ell$ whose smallest element is at least triple the maximum element of $\bigcup \mathcal{A}$.  Let $\mathcal{A}$ be a large $(p,\ell)$-good family, with $p$ arbitrary and $\ell$ large. Then the graph $G_{\mathcal{A}}$ (as defined above) does not contain three vertex-disjoint cycles whose lengths form an arithmetic progression, but $\Lambda(G_{\mathcal{A}}-X)$ contains an arithmetic progression of length $\ell$ for all $X$ with $|X| < |\mathcal{A}|$.

\paragraph{Prescribed lengths.}  Given a subset $L \subseteq \mathbb{N}$, do there exist functions $f_L : \mathbb{N} \to \mathbb{N}$ and $g_L : \mathbb{N} \to \mathbb{N}$ such that every graph $G$ contains $k$ vertex-disjoint cycles with distinct lengths from $L$ or a set $X$ of at most $f_L(k)$ vertices such that $|\Lambda(G-X) \cap L| \leq g_L(k)$?  In recent work, Gorsky, Hendrey, and Huynh~\cite{GorskyHH2026ErdosPosa} show that $f_L$ and $g_L$ do not exist if $L$ is the set of prime numbers.  More generally, they prove that $f_L$ and $g_L$ do not exist if $L$ is infinite and $\mathbb{N} \setminus L$ contains arbitrarily long sequences of consecutive integers.
\medskip

\paragraph{Directed graphs.} In a completely different direction, one can also ask if our results hold for digraphs. 
We discuss only the simplest version of this.
In particular, is it true that there exists a function $g(k)$ such that every digraph $D$ contains $k$ vertex-disjoint directed cycles of different lengths or a set $X$ of at most $g(k)$ vertices such that $D-X$ contains at most $g(k)$ different directed cycle lengths?
It turns out that this is false, as the following counterexample shows.

    Suppose there was such a function \(g(k)\).
    We construct a subdivision of a large cylindrical grid in which all cycles have the same length in the following way.
    We consider a planar embedding of a cylindrical grid such that all cycles are embedded counterclockwise around a point \(x\) in the plane and for every edge \(e\) there is a positive counterclockwise angle \(w(e) \in [360]\) between the straight line of its start-vertex to \(x\) and the straight line from its end-vertex to \(x\), see \zcref{fig:directed_counterexample}.

    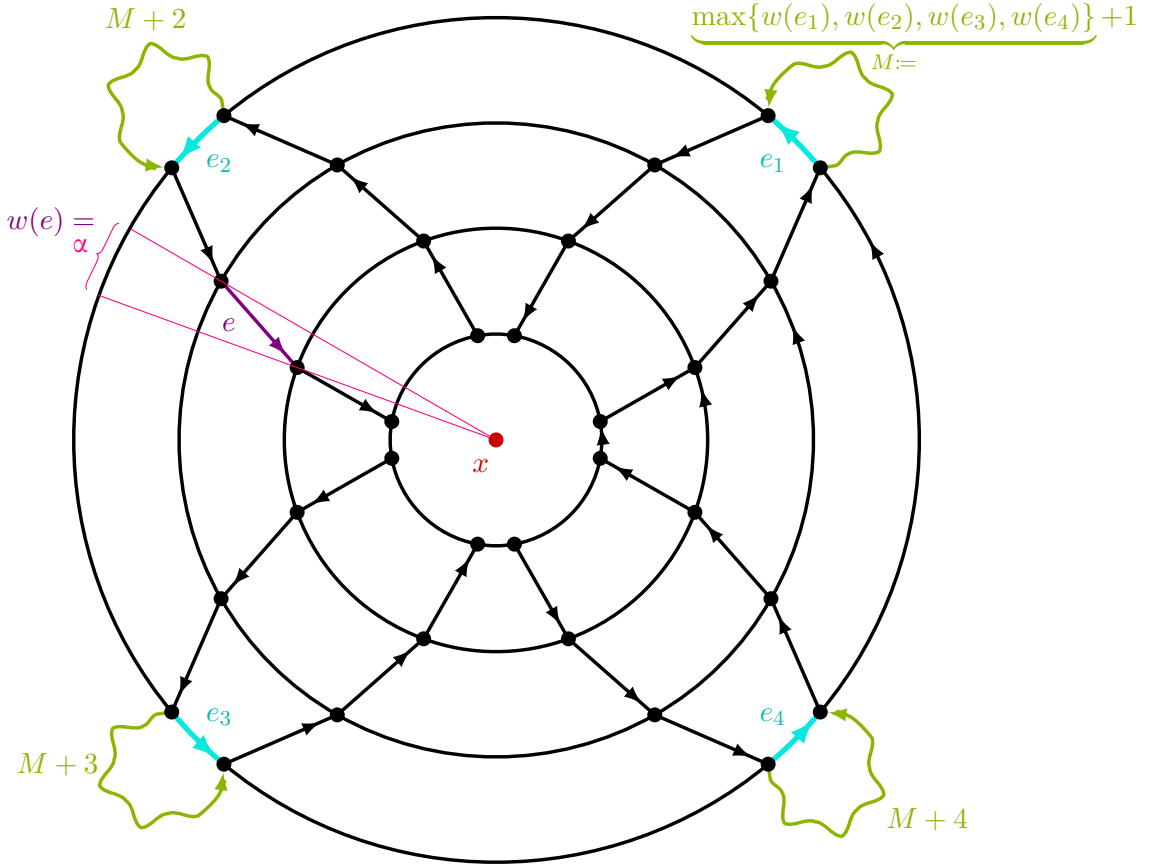
\begin{figure}[!ht]
        \begin{center}
        	\begin{tikzpicture}
        		\pgfdeclarelayer{deepbackground}
        		\pgfdeclarelayer{background}
        		\pgfdeclarelayer{foreground}
        		\pgfsetlayers{deepbackground,background,main,foreground}
        		\usetikzlibrary{arrows}
        		\usetikzlibrary{decorations.markings}
        		\usetikzlibrary{decorations.pathreplacing}
        		\tikzstyle{vertex} = [draw, circle, inner sep=.6mm, thick,fill=black]
        		\tikzstyle{edge} = [draw=black,line width=1.3pt]
        		\tikzstyle{directededge} = [edge,
        		decoration={
        			markings,
        			mark=at position 0.85 with {\arrow{latex}}},postaction={decorate}]
        		\tikzstyle{directededgesimple} = [edge,->,-latex]

        		\def\dist{1.4}
        		\node[vertex,BostonUniversityRed] (center) at (0,0) {};
        		\node[BostonUniversityRed] (center-label) at ($(center)+(240:0.4)$) {\(x\)};
        		\node[vertex] (v-1-1-out) at ($(center)+(10:\dist)$) {};
        		\node[vertex] (v-2-1-out) at ($(center)+({20}:2*\dist)$) {};
        		\node[vertex] (v-3-1-out) at ($(center)+({30}:3*\dist)$) {};
        		\node[vertex] (v-4-1-out) at ($(center)+({40}:4*\dist)$) {};
        		
        		\node[vertex] (v-4-1-in) at ($(center)+(50:4*\dist)$) {};
        		\node[vertex] (v-3-1-in) at ($(center)+({60}:3*\dist)$) {};
        		\node[vertex] (v-2-1-in) at ($(center)+({70}:2*\dist)$) {};
        		\node[vertex] (v-1-1-in) at ($(center)+({80}:\dist)$) {};
        		
        		\node[vertex] (v-1-2-out) at ($(center)+(100:\dist)$) {};
        		\node[vertex] (v-2-2-out) at ($(center)+({110}:2*\dist)$) {};
        		\node[vertex] (v-3-2-out) at ($(center)+({120}:3*\dist)$) {};
        		\node[vertex] (v-4-2-out) at ($(center)+({130}:4*\dist)$) {};
        		
        		\node[vertex] (v-4-2-in) at ($(center)+(140:4*\dist)$) {};
        		\node[vertex] (v-3-2-in) at ($(center)+({150}:3*\dist)$) {};
        		\node[vertex] (v-2-2-in) at ($(center)+({160}:2*\dist)$) {};
        		\node[vertex] (v-1-2-in) at ($(center)+({170}:\dist)$) {};
        		
        		\node[vertex] (v-1-3-out) at ($(center)+(190:\dist)$) {};
        		\node[vertex] (v-2-3-out) at ($(center)+({200}:2*\dist)$) {};
        		\node[vertex] (v-3-3-out) at ($(center)+({210}:3*\dist)$) {};
        		\node[vertex] (v-4-3-out) at ($(center)+({220}:4*\dist)$) {};
        		
        		\node[vertex] (v-4-3-in) at ($(center)+(230:4*\dist)$) {};
        		\node[vertex] (v-3-3-in) at ($(center)+({240}:3*\dist)$) {};
        		\node[vertex] (v-2-3-in) at ($(center)+({250}:2*\dist)$) {};
        		\node[vertex] (v-1-3-in) at ($(center)+({260}:\dist)$) {};
        		
        		\node[vertex] (v-1-4-out) at ($(center)+(280:\dist)$) {};
        		\node[vertex] (v-2-4-out) at ($(center)+({290}:2*\dist)$) {};
        		\node[vertex] (v-3-4-out) at ($(center)+({300}:3*\dist)$) {};
        		\node[vertex] (v-4-4-out) at ($(center)+({310}:4*\dist)$) {};
        		
        		\node[vertex] (v-4-4-in) at ($(center)+(320:4*\dist)$) {};
        		\node[vertex] (v-3-4-in) at ($(center)+({330}:3*\dist)$) {};
        		\node[vertex] (v-2-4-in) at ($(center)+({340}:2*\dist)$) {};
        		\node[vertex] (v-1-4-in) at ($(center)+({350}:\dist)$) {};
        		
        		\begin{pgfonlayer}{background}
        			%cycles
        			\draw ($(v-1-1-out)$) edge[edge,quick curve through={($(v-1-1-in)$) ($(v-1-2-out)$) ($(v-1-2-in)$) ($(v-1-3-out)$) ($(v-1-3-in)$) ($(v-1-4-out)$)}] ($(v-1-4-in)$);
%        			\draw[directededge,out=86,in=274] ($(v-1-4-in)$) to ($(v-1-1-out)$);
					\draw[directededge] ($(v-1-4-in)$) arc (-10:10:\dist);
        			
        			\draw ($(v-2-1-out)$) edge[edge,quick curve through={($(v-2-1-in)$) ($(v-2-2-out)$) ($(v-2-2-in)$) ($(v-2-3-out)$) ($(v-2-3-in)$) ($(v-2-4-out)$)}] ($(v-2-4-in)$);
%        			\draw[directededge,out=70,in=280] ($(v-2-4-in)$) to ($(v-2-1-out)$);
        			\draw[directededge] ($(v-2-4-in)$) arc (-20:20:{2*\dist});
        			
        			\draw ($(v-3-1-out)$) edge[edge,quick curve through={($(v-3-1-in)$) ($(v-3-2-out)$) ($(v-3-2-in)$) ($(v-3-3-out)$) ($(v-3-3-in)$) ($(v-3-4-out)$)}] ($(v-3-4-in)$);
%        			\draw[directededge,out=88,in=272] ($(v-3-4-in)$) to ($(v-3-1-out)$);
        			\draw[directededge] ($(v-3-4-in)$) arc (-30:30:{3*\dist});
        			
        			\draw ($(v-4-1-out)$) edge[edge,quick curve through={($(v-4-1-in)$) ($(v-4-2-out)$) ($(v-4-2-in)$) ($(v-4-3-out)$) ($(v-4-3-in)$) ($(v-4-4-out)$)}] ($(v-4-4-in)$);
%        			\draw[directededge,out=88,in=272] ($(v-4-4-in)$) to ($(v-4-1-out)$);
					\draw[directededge] ($(v-4-4-in)$) arc (-40:40:{4*\dist});

					%e_1,e_2,e_3,e_4
					\draw[directededge,line width=2,BrightTurquoise] ($(v-4-1-out)$) arc (40:50:{4*\dist});
					\node[BrightTurquoise!80!black] (e-1-label) at ($(v-4-1-out)!0.5!(v-4-1-in)+({180+45}:0.4)$) {\(e_1\)};
					
					\draw[directededge,line width=2,BrightTurquoise] ($(v-4-2-out)$) arc (130:140:{4*\dist});
					\node[BrightTurquoise!80!black] (e-2-label) at ($(v-4-2-out)!0.5!(v-4-2-in)+({180+135}:0.4)$) {\(e_2\)};
					
					\draw[directededge,line width=2,BrightTurquoise] ($(v-4-3-out)$) arc (220:230:{4*\dist});
					\node[BrightTurquoise!80!black] (e-3-label) at ($(v-4-3-out)!0.5!(v-4-3-in)+({180+225}:0.4)$) {\(e_3\)};
					
					\draw[directededge,line width=2,BrightTurquoise] ($(v-4-4-out)$) arc (310:320:{4*\dist});
					\node[BrightTurquoise!80!black] (e-4-label) at ($(v-4-4-out)!0.5!(v-4-4-in)+({180+315}:0.4)$) {\(e_4\)};
					
					%detours
%					\draw[directededgesimple,BananaYellow,bend right,decorate,decoration={snake,pre length=1mm, post length=1mm, amplitude=2mm, segment length=5mm}] ($(v-4-1-out)$) to
					%[edge,quick curve through={($(v-4-1-out)!0.5!(v-4-1-in)+({45}:0.7)$)}]
%					 ($(v-4-1-in)$);
					 
					\draw[directededgesimple,AppleGreen,decorate,decoration={snake,pre length=2mm, post length=2mm, amplitude=1mm, segment length=7mm},shorten >=1mm] ($(v-4-1-out)$) arc (-90:180:{0.7});
					\node[AppleGreen] (detour-1-label) at ($(v-4-1-in)+({28}:2.2)$) {\(\underbrace{\max\{w(e_1),w(e_2),w(e_3),w(e_4)\}}_{M\coloneqq} + 1\)};
					
					\draw[directededgesimple,AppleGreen,decorate,decoration={snake,pre length=2mm, post length=2mm, amplitude=1mm, segment length=7mm},shorten >=1mm] ($(v-4-2-out)$) arc (0:270:{0.7});
					\node[AppleGreen] (detour-2-label) at ($(v-4-2-in)+({100}:2)$) {\(M + 2\)};
					
					\draw[directededgesimple,AppleGreen,decorate,decoration={snake,pre length=2mm, post length=2mm, amplitude=1mm, segment length=7mm},shorten >=1mm] ($(v-4-3-out)$) arc (90:360:{0.7});
					\node[AppleGreen] (detour-3-label) at ($(v-4-3-in)+({180}:2.2)$) {\(M + 3\)};
					
					\draw[directededgesimple,AppleGreen,decorate,decoration={snake,pre length=2mm, post length=2mm, amplitude=1mm, segment length=7mm},shorten >=1mm] ($(v-4-4-out)$) arc (-180:90:{0.7});
					\node[AppleGreen] (detour-3-label) at ($(v-4-4-in)+({315}:2)$) {\(M + 4\)};
					
					\node at ($(center)+(180:{5.5*\dist})$) {};
					
					%paths
					\foreach \x in {1,2,3,4}{
						\foreach \i [evaluate=\i as \j using {int(\i+1)}] in {1,2,3}{
								\draw[directededge] ($(v-\i-\x-out)$) to ($(v-\j-\x-out)$);
						}
					}
					\foreach \x in {1,3,4}{
						\foreach \i [evaluate=\i as \j using {int(\i+1)}] in {1,2,3}{
							\draw[directededge] ($(v-\j-\x-in)$) to ($(v-\i-\x-in)$);
						}
					}
					
					\foreach \i [evaluate=\i as \j using {int(\i+1)}] in {1,3}{
						\draw[directededge] ($(v-\j-2-in)$) to ($(v-\i-2-in)$);
					}
					
					\draw[directededge,violet] ($(v-3-2-in)$) to ($(v-2-2-in)$);
%					\draw[directededge] ($(v-1-1-out)$) to ($(v-2-1-out)$);
        		\end{pgfonlayer}
        		
        		\node[violet] (e-label) at ($(v-3-2-in)!0.5!(v-2-2-in)+(180:0.4)$) {\(e\)};
        		
        		\draw[thin,BrightPink] (center) to ($(center)!2!(v-2-2-in)$);
        		\draw[thin,BrightPink] (center) to ($(center)!1.33!(v-3-2-in)$);
        		
        		\draw[BrightPink,decoration={brace,raise=1ex},decorate] ($(center)!2!(v-2-2-in)$) --  ($(center)!1.33!(v-3-2-in)$);
        		
        		\node[BrightPink] (alpha) at ($($(center)!2!(v-2-2-in)$)!0.5!($(center)!1.33!(v-3-2-in)$)+(155:0.5)$) {\alpha};
        		
        		\node[violet] (weight) at ($(alpha)+(140:0.5)$) {\(w(e)=\)};
        	\end{tikzpicture}
        \end{center}
        \caption{Illustration of the counterexample to the existence of a function $g$ such that every digraph contains $k$ vertex-disjoint directed cycles of different lengths or a set $X$ of at most $g(k)$ vertices such that $D-X$ contains at most $g(k)$ different directed cycle lengths.}
        \label{fig:directed_counterexample}
    \end{figure}

    Now we subdivide every edge according to the angle we just described, making it a path of length \(w(e)\).
    This way, every cycle has length exactly \(360\).

    Now, we attach many paths to the outermost cycle of the grid of such that each path begins at the end-vertex of an out-path of the grid ends at the start-vertex of the next in-path. Moreover, the lengths of the newly attached paths are all distinct and larger than 360.  

    Now there are no three disjoint cycles of different lengths.
    However, for every set \(X\) of size \(g(k)\), \(G-X\) still contains many distinct cycle lengths.  

Since the most basic form of our type of question has a negative answer, it is natural to relax its conditions and ask whether a half-integral variant, in which cycles may share vertices to a limited extent, may hold.
Such half-integral versions of the \EP{} property for directed cycles often hold, despite the integral versions failing on very similar counterexamples (see \cite{GorskyKKW2024Packing,Gorsky2024Structure,KawarabayashiKKX2025Halfintegral}).

\begin{conjecture}\label{con:halfintegraldirectedcycles}
    For every $k \geq 1$, there exists an integer $g(k)$ such that every digraph $D$ contains a set $\mathcal{C}$ of $k$ directed cycles of different lengths, such that there does not exist a vertex $v \in V(D)$ that is contained in three distinct cycles of $\mathcal{C}$, or a set $X$ of at most $g(k)$ vertices such that $D-X$ contains at most $g(k)$ different directed cycle lengths.  
\end{conjecture}

Finally, for the sake of completeness, we note that the directed version of \zcref{thm:mindegree} is famously open and not yet known to hold for any minimum out-degree condition depending only on $k$.  

\begin{conjecture}[Lichiardopol \cite{lichiardopol14}]
    For every $k \geq 1$, there exists an integer $g(k)$ such that every digraph of minimum out-degree at least $g(k)$ contains $k$ vertex-disjoint directed cycles of distinct lengths. 
\end{conjecture}

\paragraph{$H$-minor-free graphs.}

In a different direction than cycles, it is natural to ask whether \zcref{thm:colouredverticesatdistanceEP} holds more generally for $H$-minor-free graphs. That is, we leave the following as an open problem.

\begin{question}
    Let $k,d \in \mathbb{N}^+$. Let $G$ be a graph forbidding $H$ as a minor and let $c$ be a colouring of $V(G)$. Do there exist functions $f,g$ depending only on $H$ such that there exist $k$ vertices given pairwise distinct colours by $c$ and pairwise lie at distance at least $d$, or there exists a set $X \subseteq V(G)$ of order at most $f(k,d)$ such that after deleting $N^{\lfloor d/2\rfloor}(X)$, $V(G)$ has at most $g(k)$ remaining colours?
\end{question}

It is not too hard to see how to do this ``locally'' by using \zcref{thm:colouredverticesatdistanceEP}. Graphs with a forbidden $H$-minor have a tree-decomposition into parts such that every part has bounded local treewidth after deleting a small number of apices (see \cite{Grohe2003Local}). Thus, for a single part, we can recolour all vertices within distance $\lfloor d/2\rfloor$ from the apices by a single special colour and apply \zcref{thm:colouredverticesatdistanceEP}, and then we can add these apices to any hitting set. It is unclear how to extend these local solutions ``globally'' along the entire tree-decomposition.

\medskip

\paragraph{Acknowledgements:} Much of the research presented here was carried out whilst a subset of the authors were attending the BWAG workshop in Bertinoro that took place in September of 2025 and the 2026 Barbados Graph Theory Workshop.
We thank the organisers and other participants of both workshops for providing productive atmospheres to work in.

%Additionally, we thank Jim Geelen for allowing us to present his proof of \zcref{thm:hightwgivesladder} here.

\bibliographystyle{alphaurl}
\bibliography{literature}

\newpage
 
\appendix

\section{An algorithmic proof of~\zcref{thm:main}}\label{sec:algoproof}
We now present an algorithmic proof of~\zcref{thm:main}.
The key ingredient is the following theorem of Visser and Bodlaender~\cite{visserbodlaender}, which shows that the cycle spectrum of a graph can be computed in fixed-parameter time in graphs of treewidth at most $k$. 

\begin{theorem}[Visser and Bodlaender \cite{visserbodlaender}] \label{lem:computecyclespectrum}
    Let $\ell,\gamma \in \mathbb{N}$ and let $G$ be a graph with $\mathsf{tw}(G) \leq \ell$.
    Then there is an algorithm that correctly decides if $\gamma \in \Lambda(G)$, which runs in time $O(18^{\ell}\gamma^4|V(G)|)$. 
\end{theorem}

% \begin{lemma}\label{lem:computecyclespectrum}
%     Let $k \in \mathbb{N}$ and let $G$ be a graph with $\mathsf{tw}(G) \leq k$.
%     Then the set $\Lambda(G)$ can be computed in time $2^{\mathcal{O}(k^2)}|V(G)|^4$.
% \end{lemma}

If one does not care about the dependence on $\ell$, then~\zcref{lem:computecyclespectrum} also follows from a version of Courcelle's Theorem~\cite{courcellemosbah}, which can output the set of all possible sizes of a set $X$ satisfying an MSO-formula $\varphi$ on bounded treewidth graphs.

% The proof of \zcref{lem:computecyclespectrum} uses dynamic programming on tree-decompositions.  This is rather tedious and will occupy the majority of this section.  To shorten the exposition, we leave out standard arguments for the correctness of the dynamic programming updates we will present.

% We lay out most of the details involved after the proof of the next lemma, since the process for computing the cycle spectrum is quite technical.

% Assuming \zcref{lem:computecyclespectrum}, it is not hard to prove an algorithmic version of \zcref{lem:lowtwdistinctlengthep}. 
% However, we will also need to \textsl{find} representative cycles for all the lengths in the spectrum of a graph with bounded treewidth and a small spectrum in said proof.
% This could be done in the proof of \zcref{lem:computecyclespectrum} itself, as a part of the dynamic program, but would increase the complexity of the description of said dynamic program to a level which we did not find particularly instructive.
% Therefore, we opt for a slightly more brute-force approach to resolve this problem at the cost of a somewhat worse runtime for the algorithmic version of \zcref{lem:lowtwdistinctlengthep} we are about to present.

We also need to efficiently find cycles of each length in the cycle spectrum of a graph with bounded treewidth.
The algorithm of Visser and Bodlaender~\cite{visserbodlaender} does not output these cycles, although it can probably be adapted to do so.  However, the following naive approach works, at the cost of a slightly worse runtime.

\begin{lemma}
    \label{lem:computecycleofgivenlength}
    Let~${\ell \in \mathbb{N}}$ and let~$G$ be a graph with~${\mathsf{tw}(G) \leq \ell}$. 
    Suppose~${\gamma \in \Lambda(G)}$. 
    Then we can find a cycle of length~$\gamma$ in time~$O(\ell18^{\ell}\gamma^5|V(G)|^2)$. 
\end{lemma}
\begin{proof}
    If $|E(G)|=\gamma$, then $E(G)$ is the set of edges of a cycle of length $\gamma$, so we are done.  Thus, we may assume that $|E(G)| > \gamma$.  By applying \zcref{lem:computecyclespectrum} to~${G - e}$ for $\gamma+1$ arbitrary  edges $e$ of $G$ we can find in time $O(18^{\ell}\gamma^4|V(G)|(\gamma+1))$  an edge~${e \in E(G)}$ for which~${\gamma \in \Lambda(G-e)}$. 
    Iterating this $|E(G)|-\gamma$ times, we reduce to the case that $|E(G)|=\gamma$.  Thus, in time $O(18^{\ell}\gamma^4|V(G)|(\gamma+1) |E(G)|)\subseteq O(\ell18^{\ell}\gamma^5|V(G)|^2)$ we can find a cycle of length~$\gamma$ in~$G$. 
\end{proof}

    % Let~${n \coloneqq |V(G)|}$. 
    % If~${n = \ell}$, let~${G' \coloneqq G}$. 
    % Suppose~${n > \ell}$. 
    % By applying \zcref{lem:computecyclespectrum} to~${G - v}$ for each vertex~${v \in V(G)}$, we can find in time~$2^{\mathcal{O}(k^2)} n^5$ a vertex~${v \in V(G)}$ for which~${\ell \in \Lambda(G-v)}$. 
    % Iterating this process~${n - \ell}$ times, we find in time~$2^{\mathcal{O}(k^2)}n^6$ a graph~$G'$ with~$\ell$ vertices and such that~${\ell \in \Lambda(G')}$. 
    % Since~$G'$ still has treewidth at most~$k$, $G'$ has fewer than~${\ell k}$ edges. 
    % We apply \zcref{lem:computecyclespectrum} to~${G' - e}$ for each edge~${e \in E(G)}$ to find in time~$2^{\mathcal{O}(k^2)} \ell^5$ an edge~${e \in E(G)}$ for which~${\ell \in \Lambda(G'-e)}$. 
    % Iterating this process at most~${k\ell - \ell}$ times, we find in time~$2^{\mathcal{O}(k^2)}\ell^6$ a cycle of length~$\ell$ in~$G'$. 
    % Since~$\ell \leq n$, we can find the desired cycle in time $2^{\mathcal{O}(k^2)} n^6$. 

Finally, we also require the following elementary fact about (rooted) tree-decompositions.

\begin{lemma}\label{lem:treedecompseparators}
    Let $G$ be a graph with a rooted tree-decomposition $(T,r,\beta)$.
    Then for every vertex ${t \in V(T)}$, within ${G - \beta(t)}$ there is no path from any vertex in $V(G_t - \beta(t))$ to any vertex in $\bigcup_{s \in V(T) \setminus V(T_t)} \beta(s)$.
\end{lemma}

This brings us to the promised algorithmic version of \zcref{lem:lowtwdistinctlengthep}.

\begin{lemma}\label{lem:lowtwdistinctlengthepalgo}
    Let $k \in \mathbb{N}^+$ and let $G$ be a graph of treewidth at most $\ell$.
    Then $G$ contains $k$ pairwise vertex-disjoint cycles of pairwise distinct lengths, or there exists a set $X \subseteq V(G)$ with $|X| \leq (k-1)(\ell+1)$ such that $|\Lambda(G - X)| \leq k - 1$.
    Moreover, there exists an algorithm running in time $O(\ell 18^{\ell}\mathsf{poly}(k)|V(G)|^7)$ that finds either the $k$ cycles or the set $X$ with the desired properties.
\end{lemma}

    % In particular, there exists an algorithm running in time $2^{\mathcal{O}(\ell^2)}\mathsf{poly}(k)|V(G)|^6$ that finds either the $k$ cycles or the set $X$ with the desired properties.
\begin{proof}
    Let $(T,r,\beta)$ be a tree-decomposition of~$G$ with width at most~$\ell$ and an arbitrarily chosen root~$r$, and let~${\lvert V(T) \rvert = a}$. 
    Let~$(t_1, \dots, t_{a})$ be an enumeration of the nodes of~$T$ such that~${t_{a} = r}$ and~$t_i$ is a leaf of $T-\{t_1, \dots, t_{i-1}\}$ for each~${i \in [a-1]}$. 
    We process the tree-decomposition in order of this enumeration to either find a desired selection of $k$ cycles, in which case we stop the process, or find the set~$X$. 

    For~$Z \subseteq V(T)$ and~$t \in V(T) \setminus Z$, let~$T(t,Z)$ denote the component of~$T_t - Z$ containing~$t$ and let~$H(t,Z)$ denote~$G[\bigcup_{s \in V(T(t,Z))} \beta(s)] - \bigcup_{s \in Z} \beta(s)$. 

    We initialise~${k_0 = 0}$. 
    Now for some~${i \in \{0, \dots, {a-1}\}}$, suppose we have already defined a non-negative integer~$k_i$ and a tuple $(z_1, \dots, z_{k_i})$ of~$k_i$ distinct elements in~$\{t_1, \dots, t_i\}$ such that
    \begin{enumerate}
        \item the tuple~$(H_1, \dots, H_{k_i})$, where~$H_j = H(z_j,\{z_1, \dots, z_{j-1}\})$, is a tuple of pairwise vertex-disjoint subgraphs, 
        \item the set~$\mathcal{P}_i$ of all $k_i$-tuples in~$\Lambda(H_1) \times \dots \times \Lambda(H_{k_i})$ with~$k_i$ distinct elements is non-empty.
    \end{enumerate}
    Now consider~${H' \coloneqq H(t_{i+1}, \{z_1, \dots, z_{k_i}\})}$. 
    If there exists~${x \in \Lambda(H')}$ and~${(x_1,\dots,x_{k_i}) \in \mathcal{P}_i}$ such that~${x \neq x_j}$ for all~${j \in [k_i]}$, then we set~${k_{i+1} \coloneqq k_i + 1}$ and~${z_{k_{i+1}} \coloneqq t_{i+1}}$. 
    Then, clearly, the second property is satisfied. For the first property, observe that by \zcref{lem:treedecompseparators}, $H'$ is vertex-disjoint from each of~$H_1, \dots, H_{k_i}$. 
    Otherwise, we set~${k_{i+1} \coloneqq k_i}$. 

    Note that we do not have to store $\mathcal{P}_i$ directly, which may have size $\mathcal{O}(n^k)$ if each $\Lambda(H_j)$ has size $n$. We can instead store $\Lambda(H_j)$ for each $j \in [k_i]$. Let $n=|V(G)|$. To find whether there exists $x \in \Lambda(H')$ and $x_1, \dots, x_{k_i}$ with $x_j \in \Lambda(H_j)$ all pairwise distinct, we construct a bipartite graph with one side $\{3,4,\dots,n\}$ and the other side $\Lambda(H_1), \dots, \Lambda(H_{k_i}), \Lambda(H')$, where a set is adjacent to the lengths it contains.
     A matching in this graph covering the sets gives the desired distinct cycle lengths.

    We observe that for each~${i \in [a]}$, by construction, the union~$H_1 \cup \dots \cup H_{k_i} \subseteq G$ contains $k_i$ pairwise vertex-disjoint cycles of pairwise distinct length, and we can find these cycles in the desired time by \zcref{lem:computecycleofgivenlength}.
    So if~${k_{a} \geq k}$, we are done. 
    So suppose~${k_{a} \leq k-1}$. 
    We claim that the set~${X \coloneqq \beta(z_1) \cup \dots \cup \beta(z_{k_{a}})}$ is as desired. 
    Indeed, the set has size at most~${(k-1)(\ell+1)}$. 
    Suppose towards a contradiction that~${\lvert \Lambda(G - X) \rvert \geq k}$. 
    Then, in particular, for each~${i \in [a]}$ and each~${(x_1, \dots, x_{k_i}) \in \mathcal{P}_i}$ there exists a cycle~$C$ in~$G - X$ whose length is distinct from~${x_1, \dots, x_{k_i}}$. 
    Let~$j$ be maximal such that~${C \subseteq G_{t_j} - X}$. 
    In particular, $C$ intersects~$\beta(t_j)$. 
    But then, $C$ is a subgraph of~$H(t_j,\{z_1,\dots,z_{k_{j-1}}\})$. 
By construction, $t_j$ would have been added to the tuple $(z_1,\dots,z_{k_a})$ when it was processed, contradicting that $C$ is disjoint from $X$.
\end{proof}

By combining~\zcref{lem:lowtwdistinctlengthepalgo} with~\zcref{cor:hightwcyclepacking}, we obtain an algorithmic version of~\zcref{thm:main} with runtime $\mathcal{O}(18^{k^5} \mathsf{poly}(k) |V(G)|^7)$, as claimed.  

\section{Proof of~\zcref{lem:bndtwcolouredverticesatdistanceEP}}\label{sec:facesapp}
Here we present the proof of \zcref{lem:bndtwcolouredverticesatdistanceEP}.  We use the following state-of-the-art algorithm to compute a tree-decomposition of width at most $k$.

\begin{theorem}[Korhonen and Lokshtanov \cite{KorhonenL2023Improved}]\label{thm:computetwexact}
    There exists an algorithm that takes as input a graph $G$ and an integer $k$, and in time $2^{\mathcal{O}(k^2)}|V(G)|^4$ either outputs a tree-decomposition of $G$ of width at most $k$ or concludes that the treewidth of $G$ is larger than $k$.
\end{theorem}

% The techniques used here are very similar to what we present in \zcref{sec:algoproof}, though we have to add a few more tricks to adjust to the specific setting of this statement.
% We note that finding the cycles of specific colours is much easier in this particular case, causing our runtime to drop compared to \zcref{lem:lowtwdistinctlengthepalgo}.

\begin{replemma}{lem:bndtwcolouredverticesatdistanceEP}
    Let $k,d \in \mathbb{N}^+$. Let $G$ be a graph with $\mathsf{tw}(G) \leq \ell$, and let $c$ be a colouring of $V(G)$. Then there exist $k$ vertices that are given pairwise distinct colours by $c$ and pairwise lie at distance at least $d$, or there exists a set $X \subseteq V(G)$ with $|X| \leq (k-1)(\ell + 1)$ such that $|\{c(v) \mid v \in V(G - N^{\lfloor d/2\rfloor}(X))\}| \leq k-1$.

    In particular, there exists an algorithm running in time $2^{\mathcal{O}(\ell^2)}\mathsf{poly}(k,d,\ell)|V(G)|^4$ that finds either the packing of $k$ vertices or the set $X$ with the properties promised above.
\end{replemma}
\begin{proof}
    We may assume that $d$ is odd; otherwise, we apply the result to $d+1$. Hence, for $u,v \in V(G)$, $u$ and $v$ are at distance at least $d$ if and only if $N^{\lfloor d/2\rfloor}(u) \cap N^{\lfloor d/2\rfloor}(v) = \emptyset$. Therefore, to get a short non-algorithmic proof of the lemma, we may proceed as in the proof of \zcref{lem:lowtwdistinctlengthep} using $N^{\lfloor d/2\rfloor}(v)$ instead of the cycles. We now give an algorithmic proof analogous to \zcref{lem:lowtwdistinctlengthepalgo}.

    Let $(T,r,\beta)$ be a rooted tree-decomposition of $G$ with width at most $\ell$. Let $|V(T)| = a$ and let $(t_1, \dots, t_a)$ be an enumeration of the nodes of $T$ such that $t_a = r$ and $t_i$ is a leaf of $T - \{t_1, \dots, t_{i-1}\}$ for each $i \in [a-1]$. We process the tree-decomposition in order of this enumeration to either find a desired selection of $k$ vertices, in which case we stop the process, or the set $X$.

    For an induced subgraph~$H$ of~$G$, let~$C(H)$ denote the set~$\{ c(v) \mid N^{\lfloor d/2 \rfloor}_G(v) \subseteq V(H)\}$.

    For~$Z \subseteq V(T)$ and~$t \in V(T) \setminus Z$, let~$T(t,Z)$ denote the component of~$T_t - Z$ containing~$t$ and let~$H(t,Z)$ denote~$G[\bigcup_{s \in V(T(t,Z))} \beta(s)] - \bigcup_{z \in Z}\beta(z)$. 

    We initialise~${k_0 = 0}$. 
    Now for some~${i \in \{0, \dots, {a-1}\}}$, assume that we have already defined a non-negative integer~$k_i$ and a tuple $(z_1, \dots, z_{k_i})$ of~$k_i$ distinct elements in~$\{t_1, \dots, t_i\}$ such that
    \begin{enumerate}
        \item the tuple~$(H_1, \dots, H_{k_i})$, where~$H_j = H(z_j,\{z_1, \dots, z_{j-1}\})$, is a tuple of vertex-disjoint subgraphs, 
        \item the set~$\mathcal{P}_i$ of all $k_i$-tuples in~$C(H_1) \times \dots \times C(H_{k_i})$ with~$k_i$ distinct elements is non-empty. 
    \end{enumerate}
    Now consider~${H' \coloneqq H(t_{i+1}, \{z_1, \dots, z_{k_i}\})}$. 
    If there exists~${x \in C(H')}$ and~${(x_1,\dots,x_{k_i}) \in \mathcal{P}_i}$ such that~${x \neq x_j}$ for all~${j \in [k_i]}$, then we set~${k_{i+1} \coloneqq k_i + 1}$ and~${z_{k_{i+1}} \coloneqq t_{i+1}}$. 
    Then, clearly, the second property is satisfied.
    For the first property, observe that by \zcref{lem:treedecompseparators}, $H'$ is vertex-disjoint from each of~$H_1, \dots, H_{k_i}$.
    Otherwise, we set~${k_{i+1} \coloneqq k_i}$. 

    Note that, as in the proof of \zcref{lem:lowtwdistinctlengthepalgo}, we do not need to store~$\mathcal{P}_i$ directly. 
    Storing only each individual~$C(H_1), \dots, C(H_{k_i})$ allows us to recover an element of~$\mathcal{P}_i$ by finding a matching in the bipartite graph with colours on one side, the sets~$C(H_1), \dots, C(H_{k_i})$ on the other side, and with an edge between a colour~$a$ and~$C(H_j)$ if~${a \in C(H_j)}$.

    We observe that for each~${i \in [a]}$, by construction, the union~$H_1 \cup \dots \cup H_{k_i} \subseteq G$ contains a packing of vertices as desired. 
    So if~${k_{a} \geq k}$, we are done. 
    So suppose~${k_{a} \leq k-1}$. 
    We claim that the set~${X \coloneqq \beta(z_1) \cup \dots \cup \beta(z_{k_{a}})}$ is as desired. 
    Indeed, the set has size at most~${(k-1)(\ell+1)}$. 
    Suppose towards a contradiction that~${|\{c(v) \mid v \in V(G - N^{\lfloor d/2\rfloor}(X))\}| \geq k}$. 
    Then, in particular, for each~${i \in [a]}$ and each~${(x_1, \dots, x_{k_i}) \in \mathcal{P}_i}$ there exists a vertex $v \in V(G - N^{\lfloor d/2\rfloor}(X))$ with~$c(v) \notin \{x_1, \dots, x_{k_i}\}$. 

    Let~$j \in [a]$ be minimal such that~${v \in V(G_{t_j} - N^{\lfloor d/2 \rfloor}(X))}$. 
    In particular, $v \in \beta(t_j)$. 
    But then, $N^{\lfloor d/2\rfloor}_G(v)$ is a subset of~$V(H(t_j,\{z_1,\dots,z_{k_{j-1}}\}))$. 
    By construction, $t_j$ would have been added to the tuple $(z_1,\dots,z_{k_a})$ when it was processed, contradicting that $v \notin N^{\lfloor d/2\rfloor}_G(X)$.

    By \zcref{thm:computetwexact}, we can find an optimal tree-decomposition in $2^{\mathcal{O}(\ell^2)}|V(G)|^4$ time. Then, for each bag in the tree, we may need to compute $C(H(t,Z))$ for some particular $Z$. We can use breadth-first search to determine whether a vertex $v$ has $N^{\lfloor d/2\rfloor}_G(v) \subseteq V(H)$, and hence we can compute $C(H(t,Z))$ in $\mathcal{O}(|V(G)|^2)$ time. We then may have to solve a matching problem in a bipartite graph with $\mathcal{O}(|V(G)|)$ vertices, which can be done in $\mathcal{O}(|V(G)|^{3})$ time (and in fact much faster). Thus the total runtime is bounded by $2^{\mathcal{O}(\ell^2)}\mathsf{poly}(k,d,\ell)|V(G)|^4$.%\CM{it's unclear to me whether we can drop the $\mathsf{poly}(k,d,\ell)$.}\MG{Same. But I would suggest to just keep it. We are doing FPT-stuff here anyway.}
\end{proof}

\end{document}